\documentclass[9pt,shortpaper,twoside,web]{ieeecolor} 
\let\labelindent\relax

\usepackage{lcsys}
\usepackage{cite} 
\usepackage{amsmath,amssymb,amsthm,upgreek,graphics,epsfig,times,enumitem,float,mathtools,mathrsfs,bm}
\usepackage{chngcntr}
\usepackage{apptools}
\AtAppendix{\counterwithin{lemma}{section}}
\usepackage[dvipsnames]{xcolor}
\usepackage{hyperref} 
\hypersetup{colorlinks=true, linkcolor=black, filecolor=black, urlcolor=black,citecolor=black}
\usepackage[linesnumbered,ruled]{algorithm2e}
\usepackage[caption=false]{subfig} 
\usepackage{times,mathptmx}
\usepackage{authblk}
\usepackage{tikz}
\usetikzlibrary{arrows}
\usetikzlibrary{shapes}
\usetikzlibrary{arrows.meta} 
\usetikzlibrary{calc,intersections,through,backgrounds}
\usetikzlibrary{positioning}
\usepackage{pgfplots}
\pgfplotsset{compat=1.16}
\usepackage{pgfplotstable}
\usepackage{filecontents}
\usepackage{multirow}
\theoremstyle{theorem}
\newtheorem{theorem}{Theorem} 
\newtheorem{corollary}{Corollary} 
\newtheorem{lemma}{Lemma} 
\newtheorem{proposition}{Proposition} 
\theoremstyle{definition}
\newtheorem{definition}{Definition} 
\newtheorem{assumption}{Assumption} 
\newtheorem{remark}{Remark} 
 
\newcommand\normx[1]{\Vert#1\Vert}

\newcommand{\R}{\mathbb{R}}
\newcommand{\N}{\mathsf{N}}
\newcommand{\K}{\mathsf{K}}
\newcommand{\U}{\mathsf{U}}

\newcommand{\X}{\mathsf{X}}
\newcommand{\bu}{\mathsf{u}}

\newcommand{\col}[1]{\mathrm{col}\{#1\}}

\usepackage{textcomp}
\def\BibTeX{{\rm B\kern-.05em{\sc i\kern-.025em b}\kern-.08em
    T\kern-.1667em\lower.7ex\hbox{E}\kern-.125emX}}
\begin{document}

\title{Generalized open-loop Nash equilibria in linear-quadratic difference games with coupled-affine inequality constraints}
\author{Partha Sarathi Mohapatra and Puduru Viswanadha Reddy, \IEEEmembership{Member, IEEE} 
		\thanks{P. S. Mohapatra and P. V. Reddy are with the Department of Electrical Engineering, Indian Institute of Technology-Madras, Chennai, 600036, India.
		{(e-mail:   ps\textunderscore mohapatra@outlook.com, vishwa@ee.iitm.ac.in)}}} 	
	\date{ }
	\maketitle
        \thispagestyle{headings}
	\begin{abstract}  
		In this note, we study a class of deterministic finite-horizon linear-quadratic difference games with coupled affine inequality constraints involving both state and control variables. We show that the necessary conditions for the existence of generalized open-loop Nash equilibria in this game class lead to two strongly coupled discrete-time linear complementarity systems. Subsequently, we derive sufficient conditions by establishing an equivalence between the solutions of these systems and convexity of the players' objective functions. These conditions are then reformulated as a solution to a linear complementarity problem, providing a numerical method to compute these equilibria. We illustrate our results using a network flow game with constraints.
	\end{abstract}
\section{Introduction}   
Dynamic game theory (DGT) provides a mathematical framework for analyzing multi-agent decision processes evolving over time. DGT has been applied effectively in engineering, management science, and economics, where dynamic multi-agent decision problems arise naturally (see \cite{isaacs:65, Engwerda:05, Basar:99, Dockner:00}). Notable engineering applications include cyber-physical systems \cite{Zhu:15}, communication and networking \cite{Zazo:16}, and smart grids \cite{Chen:16}. Most existing DGT models in these works are formulated in unconstrained settings. However, real-world multi-agent decisions often involve constraints like saturation limits, bandwidth restrictions, production capacities, budgets, and emission limits. These introduce equality and inequality constraints into the dynamic game model, linking each player's decisions to others' at every stage. As a result, players' decision sets are interdependent, often termed \emph{coupled} or non-rectangular.

In static games, where players make decisions only once, the generalized Nash equilibrium (GNE) extends the Nash equilibrium concept to scenarios with interdependent or non-rectangular decision sets (see \cite{Debreu:54}, \cite{Rosen:65}). The existence conditions for GNE in general settings have been studied in operations research using variational inequalities (see \cite{Harker:91} and \cite{Facchinei:07}). Recently, dynamical systems-based methods for computing GNE, known as GNE-seeking, have gained attention for controlling large-scale networked engineering systems; see the recent review article \cite{Belgioioso:22}   for more details. 

 This note focuses on the existence conditions for generalized Nash equilibria in dynamic games where players' decision sets are coupled at each stage. Unlike static games, Nash equilibrium strategies in dynamic games vary based on the information structure available to players. In the open-loop structure, decisions depend on time and the initial state, while in the feedback structure, they depend on the current state. Linear quadratic (LQ) differential games with implicit equality constraints modeled by differential-algebraic equations (DAEs) have been studied in \cite{Engwerda:09, Engwerda:12, Reddy:13} for both open-loop and feedback information structures. Stochastic formulations of DAE-constrained LQ differential games with feedback structures were investigated in \cite{Tanwani2}. Scalar LQ multi-stage games arising in resource extraction with state-dependent upper bounds on decision variables were examined in \cite{Singh2}. In \cite{Laine:23}, the authors studied linear and nonlinear difference games with equality and inequality constraints, as well as numerical methods for approximating generalized feedback Nash equilibria. The authors in \cite{Reddy:15, Reddy:17} examined LQ difference games with affine inequality constraints, providing conditions for constrained open-loop and feedback Nash equilibria. These games involve two types of control variables: one influencing state evolution independently and the other affecting constraints, with the former not coupled. A class of difference games with inequality constraints, termed dynamic potential games, was studied in \cite{Zazo:16}. Restrictive conditions in this study allowed for computing generalized open-loop Nash equilibria by solving a constrained optimal control problem. However, these conditions limit the range of possible strategic scenarios. In \cite{Partha:23}, scalar mean-field LQ difference games with coupled affine inequality constraints were explored, providing conditions for mean-field-type solutions. In summary, to the best of our knowledge, both necessary and sufficient conditions for the existence of generalized open-loop Nash equilibrium (GOLNE) strategies are not yet available for deterministic LQ difference games where players' control variables are fully coupled through inequality constraints.

\subsubsection*{Contributions}  
We consider a deterministic, finite-horizon, non-zero-sum LQ difference game with inequality constraints, where players' decision sets at each stage are coupled by affine inequality constraints involving state and control variables. This note aims to derive both necessary and sufficient conditions for the existence of GOLNE strategies in this class of dynamic games. Our contributions are as follows: 
\begin{enumerate}
	\item In subsection \ref{sec:Necessary}, we show that the necessary conditions for GOLNE result in a discrete-time coupled linear complementarity system (LCS).
	\item In Theorem \ref{thm:SufficentRiccati}, we derive an auxiliary static game with coupled constraints from the dynamic game, demonstrating that the necessary conditions for a generalized Nash equilibrium in this static game lead to another discrete-time LCS.
	\item  In Theorems \ref{th:TPBlink} and \ref{th:TPBlink2} and Corollary \ref{cor:uequiv}, we establish an equivalence between the solutions of these LCSs. Using an existence result from static games with coupled constraints \cite{Rosen:65}, we provide sufficient conditions for GOLNE in Theorem \ref{thm:Nashexistence}.
	\item Assuming additional conditions on the problem data, in Theorem \ref{th:OLNElcp}, we reformulate the LCS associated with the necessary conditions into a linear complementarity problem. We then show that the joint GOLNE strategy profile is an affine function of the Lagrange multipliers associated the constraints.
\end{enumerate}

\subsubsection*{Novelty and Differences with Previous Literature}  
The novelty and contribution of this note lie in generalizing the restrictive class of LQ difference games studied in \cite{Reddy:15, Reddy:17}. Unlike those works, which consider independent or rectangular decision variables influencing state evolution, this paper addresses fully coupled decision variables. This coupling results in necessary conditions for GOLNE leading to two strongly coupled discrete-time linear complementarity systems. Our approach, as shown in Theorems \ref{th:TPBlink} and \ref{th:TPBlink2} and Corollary \ref{cor:uequiv}, analyzes these systems and establishes their equivalence, which we use to derive sufficient conditions. Our results also extend \cite{Abou:03}, which dealt with LQ difference games without constraints. Unlike \cite{Zazo:16}, we impose no structural assumptions on state dynamics and quadratic objective functions, deriving both necessary and sufficient conditions for GOLNE, while \cite{Zazo:16} offers only sufficient conditions with additional restrictions. Finally, our work differs from \cite{Engwerda:09, Engwerda:12, Reddy:13, Tanwani2} by coupling control variables through inequality constraints.

This note is organized as follows: In Section \ref{sec:Preliminaries}, we introduce LQ difference games with coupled-affine inequality constraints. The necessary and sufficient conditions for GOLNE are presented in subsections \ref{sec:Necessary} and \ref{sec:sufficientconditions}, respectively. The existence conditions for GOLNE are reformulated as a large-scale linear complementarity problem in subsection \ref{sec:Solvability}. Section \ref{sec:Numerical} illustrates our results with a network-flow game with capacity constraints. Conclusions are provided in Section \ref{sec:Conclusions}. 

\subsubsection*{Notation}\label{sec:notation}  
We denote the sets of natural numbers, real numbers, $n$-dimensional Euclidean space, $n$-dimensional non-negative orthant, and $n \times m$ real matrices by $\mathbb{N}$, $\mathbb{R}$, $\mathbb{R}^n$, $\mathbb{R}^n_{+}$, and $\mathbb{R}^{n \times m}$, respectively. The transposes of a vector $a$ and a matrix $A$ are denoted by $a^{\prime}$ and $A^{\prime}$. The quadratic form $a^{\prime}Aa$ is denoted by $\normx{a}_{A}$. For $A \in \mathbb{R}^{n \times n}$ and $a \in \mathbb{R}^n$, where $n = n_1 + \cdots + n_K$, $[A]_{ij}$ denotes the $n_i \times n_j$ submatrix, and \([a]_i\) denotes the $n_i \times 1$ subvector. Column vectors $[v_1^{\prime}, \cdots, v_n^{\prime}]^{\prime}$ are written as $\col{v_1, \cdots, v_n}$ or $\col{v_k}_{k=1}^n$. The identity and zero matrices of appropriate dimensions are denoted by $\mathbf{I}$ and $\mathbf{0}$, respectively. The block diagonal matrix with diagonal elements $M_1, \cdots, M_K$ is denoted by $\oplus_{k=1}^{K}M_k$. The Kronecker product is $\otimes$. Vectors $x, y \in \mathbb{R}^n$ are complementary if $x \geq 0$, $y \geq 0$, and $x^{\prime}y = 0$, denoted by $0 \leq x \perp y \geq 0$.
\section{Dynamic game with coupled constraints}\label{sec:Preliminaries}
We denote the set of players by $\N = \{1, 2, \cdots, N\}$, the set of time instants or decision stages by $\K = \{0, 1, ..., K\}$ for $N,K\in \mathbb N$. We define  the following two sets as $\K_l:=\K\setminus \{K\}$ and $\K_r:=\K\setminus \{0\}$.  At each time instant $k \in \K_l$, each player $i \in \N$ chooses an action $u^i_k\in \R^{m_i}$ and influences the evolution of state variable $x_k\in \mathbb R^{n}$ according to the following discrete-time linear dynamics
\begin{subequations}\label{eq:DGC}
	\begin{align}
		x_{k+1} = A_k x_k+\sum_{i\in\N}B^i_k u_k^i=A_k x_k+B_k\bu_k, \label{eq:state}
	\end{align}
	where $A_k \in \R^{n\times n}, B_k^{i}\in \R^{n\times m_i}$, $B_k:=[B_k^{1},\cdots, B_k^{N}]$,  and $\bu_k:=\col{u_k^{i}}_{i=1}^{N}\in \R^m$  ($m=\sum_{i\in\N}m_i$), with a given initial condition $x_0\in\R^n$. We further assume that these decision variables for each player $i\in\N$ at every  $k \in \K_l$ satisfy the following mixed coupled-affine inequality  {constraints} 
	\begin{align}
		M_k^ix_k+N_k^i \bu_k+r_k^i \geq 0, \label{eq:constraints}
	\end{align}
	where $M_k^{i} \in \R^{c_i\times n}, N_k^{i}\in \R^{c_i\times m}, r_k^{i}\in \R^{c_i}$. For player $i\in\N$ we denote $-i:=\N\setminus\{i\}$.
	At any instant $k\in \K_l$ the collection of actions of all players except player $i$ be denoted by $u^{-i}_k:=\col{u_k^1,\cdots,u_k^{i-1},u_k^{i+1},\cdots,u_k^N}$. The profile of actions, also referred to as a strategy, of player $i\in \N$ be denoted by $\bu^{i}:=\col{u_k^{i}}_{k=0}^{K-1}$, and the  {strategies}  of all players except player $i$ be denoted by $\bu^{-i}:=\col{u_k^{-i}}_{k=0}^{K-1}$. Each player $i\in \N$, using his strategy $\mathsf {u}^i$, seeks to minimize the following interdependent stage-additive cost functional
	\begin{multline}\label{eq:objective1}
		 {J^i}\big( \bu^i, \bu^{-i} \big)  =
		\tfrac{1}{2}\normx{x_{K}}_{Q^i_{K}}+{p_K^{i}}^\prime x_K\\
		+\tfrac{1}{2}\sum_{k\in\K_l} \big(\normx{x_k}_{Q^i_k}+2{p_k^i}^\prime x_k+\sum_{j\in\N}\normx{u_k^{j}}_{R^{i j}_k}\big), 
	\end{multline}
\end{subequations}
where $R^{i j}_{k}\in  {\R^{m_i \times m_j}}$, $R^{i i}_{k}={R_k^{ii}}^\prime$, for $i,j \in \N$, $k \in \K_l$ and $Q^i_{k}\in \R^{n \times n},\, Q^i_{k} ={Q^{i}_k}^\prime,\, p^i_{k}\in \R^{n}$ for $k \in \K$.  
Due to linear dynamics, coupled constraints and interdependent quadratic objectives, \eqref{eq:DGC} constitutes a $N$-player finite-horizon  non-zero-sum linear-quadratic difference game with coupled inequality constraints, which we refer to as DGC for the  {remainder} of the paper.
\section{Generalized Open-loop Nash Equilibrium} \label{sec:Openloop}

In this section we derive both the necessary and sufficient  conditions for the existence of  {generalized} open-loop Nash equilibrium for DGC. First, we define the  admissible strategy spaces for the players and state the required assumptions. 
Eliminating the state variable in \eqref{eq:constraints}  using \eqref{eq:state}, and collecting
the constraints \eqref{eq:constraints} of all the players, the joint constraints  at stage $k\in \K_l$ are given by
\begin{multline}
	\overline{\mathbf{M}}_k (A_{k-1}\cdots A_{0} x_0+A_{k-1}\cdots A_{1} B_{0}\bu_{0}+\cdots\\
	+ A_{k-1}B_{k-2}\bu_{k-2}+ B_{k-1}\bu_{k-1})
	+\overline{\mathbf{N}}_k \bu_k+\mathbf{r}_k \geq 0,\label{eq:vectorC1}
\end{multline}
where $\overline{\textbf{M}}_k=\col{M_k^i}_{i=1}^{N}\in\R^{c\times n}$, $\overline{\textbf{N}}_k=\col{N_k^i}_{i=1}^{N}\in\R^{c\times m}$,  $\textbf{r}_k=\col{r_k^i}_{i=1}^N\in\R^{c}$ and $c=\sum_{i\in\N}c_i$. We define the set
\begin{align}
	\Omega:=\{(x_0,(\bu^i,\bu^{-i})) \in \mathbb R^n\times \mathbb R^{Km}~|~ \eqref{eq:vectorC1} \text{ holds } \forall k\in \K_l\}.
	\label{eq:sigma}
\end{align}
The set of initial conditions for which the constraints \eqref{eq:vectorC1} are feasible is then given by
\begin{align}
	\mathsf X_0:=\{x_0\in \mathbb R^n~|~\Omega \neq \emptyset\}. \label{eq:feasiableinit}
\end{align}
Clearly,  $\Omega\neq \emptyset$ implies $\mathsf X_0\neq \emptyset$. For any $x_0\in \mathsf X_0$, the joint admissible strategy space of the players  is given by
\begin{align}
	R(x_0):=\{&(\bu^{i}, \bu^{-i})\in \R^{Km}~|~(x_0,(\bu^i,\bu^{-i}))\in \Omega\}.
	\label{eq:CConstraint}
\end{align} 
For any $x_0\in \mathsf X_0$, the admissible strategy space of Player $i\in \N$ is coupled with decisions taken by other players $\bu^{-i}$, and is given by
\begin{align}
	\U^i(\bu^{-i}):=\{\bu^{i}\in\R^{Km_i}~|~(\bu^{i}, \bu^{-i}) \in R(x_0)\}.\label{eq:AdmissibleSet}
\end{align}
We have  the following assumptions regarding DGC.
\begin{assumption}\label{ass:GameAssumption}
	\begin{enumerate}[label = (\roman*)]
		\item \label{item:1} $\Omega\neq \emptyset$, and for every $x_0\in \mathsf X_0$, $R(x_0)$ is bounded.
		\item \label{item:2} The matrices $\{[N_k^i]_i, k \in \K_l, i\in\N\}$ have full rank.
		\item \label{item:3} The matrices $\{R_{k}^{ii}, k \in \K_l, i\in\N\}$ are positive definite.
	\end{enumerate}
\end{assumption}
\noindent 
If \(\Omega \neq \emptyset\), then from \eqref{eq:feasiableinit} and \eqref{eq:CConstraint}, we have that for every \(x_0 \in \mathsf{X}_0\), \(R(x_0) \neq \emptyset\). Due to the affine inequality constraints \eqref{eq:vectorC1}, \(R(x_0)\) is an intersection of \(K\) closed half-spaces and, as a result, is a closed and convex subset of \(\mathbb{R}^{Km}\) (with the usual topology on \(\mathbb{R}^{Km}\)).  So item \ref{item:1} ensures that $R(x_0)$ is a closed, convex and bounded, which is required for the existence of GOLNE. Item \ref{item:2} is required to satisfy the constraint qualification conditions. Item \ref{item:3} is a technical assumption required for obtaining individual players' controls.
\begin{remark}\label{rem:purestateC}
 Affine pure state constraints for all $k \in \K_r$ can be reformulated as mixed-type constraints \eqref{eq:constraints} using the linear state dynamics \eqref{eq:state}, and thus incorporated into our framework, provided the rank condition in Assumption \ref{ass:GameAssumption}.\ref{item:2} is met. Constraints at $k = 0$ (i.e., on $x_0$) simply further restrict the sets $\Omega$ and $\X_0$.
\end{remark}
\begin{definition}\label{def:OCNEdef} 
	An admissible strategy profile $\bu^\star \in R(x_0)$ is {a} generalized open-loop Nash equilibrium (GOLNE) for DGC if for each player $i\in \N$ the following inequality is satisfied
	\begin{align}
		J^i( \bu^{i\star}, \bu^{-i\star}) \leq J^i(\bu^i, \bu^{-i\star}),~\forall \bu^i\in 	\U^i(\bu^{-i\star}).\label{eq:OCNEdef}
	\end{align} 
\end{definition}

In a GOLNE,  each player $i\in \N$ is committed to using the course of actions at all stages, which is pre-determined according to the strategy $\bu^{i\star}$, when other players do the same.  
\subsection{Necessary   conditions}\label{sec:Necessary}
Player $i$'s   problem \eqref{eq:OCNEdef} is given by the following discrete-time constrained optimal control problem
\begin{subequations}\label{eq:dyncon}
	\begin{align}
		&\underset{\bu^{i} \in   \U^{i}(\bu^{-i\star}) }{\textrm{min}}\,J^i ( \bu^i, \bu^{-i\star}  ),\label{eq:IndvOptimization} \\
		\textrm{sub. to } &x_{k+1} = A_k x_k+B^i_k u_k^i+\sum_{j \in -i}B^j_k u_k^{j\star},~ k \in \K_l.\label{eq:stateconstraint} 
	\end{align}
	Player $i$'s feasible set $\U(\bu^{-i\star})$ depends on the GOLNE strategies of the remaining players  as follows
	\begin{align}
		M^i_k x_k+[N^i_k]_i u_k^i+\sum_{j \in -i}[N^i_k]_j u_k^{j\star}+r^i_k \geq 0,~ k \in \K_l.\label{eq:inequalityconstraint}
	\end{align}
\end{subequations}
The Lagrangian associated with Player $i$'s problem \eqref{eq:dyncon} is given by
\begin{align}\label{eq:Lagrangian}
	\mathsf L_k^i&= \tfrac{1}{2}\big(\normx{x_k}_{Q^i_k}+2{p_k^{i}}^\prime x_k+\normx{u_k^{i}}_{R^{ii}_k}+\sum_{j\in -i}\normx{u_k^{j\star}}_{R^{i j}_k}\big)\nonumber\\
	&\quad+{\lambda_{k+1}^{i}}^\prime \big(A_k x_k+B^i_k u_k^i+\sum_{j \in -i }B^j_k u_k^{j\star}\big)\nonumber\\
	&\quad-{\mu_{k}^{i}}^\prime \big(M^i_k x_k+[N^i_k]_i u_k^i+\sum_{j \in -i}[N^i_k]_j u_k^{j\star}+r^i_k\big), 
\end{align}
where  $\lambda_{k+1}^i  {\in \R^{n}}$ and $\mu_k^i {\in \R^{c_i}_+}$ be the multipliers associated with state dynamics \eqref{eq:stateconstraint} and the inequality constraints \eqref{eq:inequalityconstraint}, respectively. The necessary conditions for the existence of GOLNE for DGC are then obtained by  applying the discrete-time Pontryagin’s maximum principle \cite{Sethi:06} to the constrained optimal control problem \eqref{eq:dyncon} for each $i\in \N$, and are given by the following equations:
	\begin{subequations}\label{eq:LQKKT}
		\begin{align}
			&x_{k+1}^{\star} = A_k x_k^{\star}-\sum_{j\in\N}{B}^{j}_k (R_k^{jj})^{-1}\big({B^{j}_k}^\prime\lambda_{k+1}^j-[N^{j}_k]_j^\prime \mu_{k}^{j\star} \big), \label{eq:OPstate}\\
			&\lambda_k^i = Q_k^ix_k^{\star}+p_k^{i}+A^\prime_k\lambda_{k+1}^{i}-{M^{i}_k}^\prime \mu_{k}^{i\star}, \label{eq:OPlambda}\\ 
			&  0 \leq M_k^ix_k^{\star}-\sum_{j\in\N} [N^i_k]_j (R_k^{jj})^{-1}\big({B^{j}_k}^\prime\lambda_{k+1}^j-[N^{j}_k]_j^\prime \mu_{k}^{j\star} \big) +r_k^i\perp \mu_{k}^{i\star} \geq 0,\label{eq:OpconstraintComp} 
		\end{align}
	\end{subequations} 	
	with boundary conditions $x_0^{\star}=x_0$ and ${\lambda}^i_{K} = {Q}^i_{K}x_{K}^{\star}+{p}^i_{K}$, $i\in \N$.   Following Assumption \ref{ass:GameAssumption}.\ref{item:3}, on the positive-definiteness of $R_k^{ii}$, the GOLNE control of Player $i$  is obtained uniquely as
	\begin{align}
		u_k^{i\star}:=-(R_k^{ii})^{-1}\big({B^{i}_k}^\prime\lambda_{k+1}^i-[N^{i}_k]_i^\prime \mu_{k}^{i\star}\big),~k\in \K_l.\label{eq:OLNE} 
	\end{align} 

 Here, \eqref{eq:OPstate}-\eqref{eq:OPlambda} constitute forward and backward difference equations, and  \eqref{eq:OpconstraintComp} is a complementarity condition. As a result, the necessary conditions \eqref{eq:LQKKT} for all players constitute a coupled discrete-time linear complementarity system (LCS1); see \cite{Heemels:00} and \cite{vanderSchaft:07}. We denote solution of LCS1, if it exists, by
\begin{align} 
	\{\bm{x}^\star,\bm{\lambda},\bm{\mu}^\star\}^1:=\{{x}_k^{\star}, \lambda_k^i, ~k \in\K, ~\mu_k^{i\star}, ~k \in\K_l,~ i\in \N\}.
\end{align} 
\begin{remark}\label{rem:NashNecessay} 
   Since the conditions \eqref{eq:LQKKT} are necessary,   a solution  of LCS1, if it exists, provides a candidate GOLNE. Further, if LCS1 admits a unique solution then there exists at most one GOLNE. 
\end{remark}
\begin{remark}
The necessary conditions \eqref{eq:LQKKT} demonstrate a strong coupling between the variables $	\{\bm{x}^\star,\bm{\lambda},\bm{\mu}^\star\}^1$ due to the interdependence of the players' action sets \eqref{eq:constraints}. In contrast, in the restricted class of games studied in \cite{Reddy:15}, the necessary conditions in \cite[Theorem III.1]{Reddy:15} display weak coupling because the decision sets related to the controls affecting the state variable are independent.	
\end{remark}
\subsection{Sufficient conditions} \label{sec:sufficientconditions}
In this subsection, we derive  sufficient conditions under which the candidate equilibrium, obtained from the solution of LCS1, is indeed a GOLNE. To this end, for each Player $i$ ($i\in \N$) we define the following quadratic function 
\begin{align}
	V_k^i=\tfrac{1}{2}\normx{x_k}_{E_k^i}+{e_k^{i}}^{\prime} x_k+f_k^i, \label{eq:Vfnc}
\end{align}
where  $E_k^i \in \R^{n\times n}$, $e_k^i \in \R^{n}$, $f_k^i \in \R$ for $k\in \K$. Next, we have the following assumption on the matrices $\{E_k^i,~k\in \K,~i\in \N\}$.
\begin{assumption}\label{ass:Yassumption}
	The backward symmetric matrix Riccati equation 
    \begin{align}
		E_k^i  = A_k^{\prime} E_{k+1}^i A_k+Q^i_k-A_k^{\prime} E_{k+1}^i
		B_k^i \big(Y_k^i\big)^{-1}{B_k^{i}}^{\prime} E_{k+1}^i A_k, ~ E_K^i=Q^i_K,\label{eq:EEQ}
	\end{align}
admits a solution $\{E_k^i,~k\in \K\}$ for each $i\in \N$, and thus the matrices $\{Y_k^i=R^{ii}_k+{B_k^{i}}^\prime E_{k+1}^iB_k^i,~k \in \K_l,~i\in\N\}$ are invertible.
\end{assumption}
We note that if Assumption \ref{ass:Yassumption} holds and $\{\bm{x}^\star, \bm{\lambda}, \bm{\mu}^\star\}^1$ is a solution of LCS1, then for any strategy $(\bu^i, \bu^{-i}) \in R(x_0)$, the difference equations
	\begin{subequations}\label{eq:RiccatiBackward}
		\begin{align}
			&Y_k^i b_k^i={B_k^{i}}^\prime (E_{k+1}^i \eta_k^{i}+e_{k+1}^i\big)-[N^{i}_k]_i^\prime \mu_k^{i\star},\label{eq:bEQ}\\
			&e_k^{i} =A_k^\prime e_{k+1}^i+A_k^\prime E_{k+1}^i \eta_k^{i}-A_k^\prime E_{k+1}^i B_k^{i} b_k^i +p_k^i-{M_k^{i}}^\prime \mu_k^{i\star},\label{eq:eEQ}\\
			&f_k^i=f_{k+1}^i+\tfrac{1}{2}{\eta_k^{i}}^\prime (E_{k+1}^i\eta_k^i+2e_{k+1}^i)\nonumber\\
			&\hspace{0.5in}+\tfrac{1}{2}\sum_{j\in -i}\normx{u_k^{j}}_{R^{i j}_k}-\tfrac{1}{2}\normx{b_k^{i}}_{Y_k^i}-{\mu_{k}^{i\star}}^\prime \big(\alpha_k^{i}+r_k^i\big),\label{eq:fEQ}
		\end{align}
	\end{subequations}
with boundary conditions $e_K^{i} = p_K^i$, $b_K^i = 0$, and $f_K^i = 0$, are solvable for all $k \in \K$, where $\eta_k^i := \sum_{j \in -i} B^j_k u_k^j$ and $\alpha_k^i := \sum_{j \in -i} [N^i_k]_{j} u_k^j$.  
Using \eqref{eq:EEQ}–\eqref{eq:RiccatiBackward}, the following auxiliary result expresses the objective function \eqref{eq:objective1} in a form that will be useful in deriving the sufficient condition later in Theorem \ref{thm:Nashexistence}.
\begin{theorem}\label{thm:SufficentRiccati} 
   Let Assumption \ref{ass:Yassumption} hold. Let $\{\bm{x}^\star,\bm{\lambda},\bm{\mu}^\star\}^1$  be a solution of LCS1. Then, the objective function \eqref{eq:objective1} of each player $i\in \N$  can be expressed as:
	\begin{align}
		J^i(\bu^i, \bu^{-i})&=V_0^i+\tfrac{1}{2}\sum_{k\in\K_l}\normx{u_k^{i}+y_k^i}_{Y_k^i}\nonumber\\
		&\quad+\sum_{k\in\K_l}{\mu_{k}^{i\star}}^\prime \big(M_k^ix_k+[N_k^i]_i u_k^i+\alpha_k^i+r_k^i\big),\label{eq:CostasV}
	\end{align}
	where $y_k^i:=(Y_k^i)^{-1}{B_k^i}^\prime E_{k+1}^i A_k x_k+b_k^i$, and ${x_k,~k\in \K}$  satisfies \eqref{eq:state}.
	
\end{theorem}
\begin{proof}
		The proof is based on the direct method (also referred to as the completion of squares). We outline the steps involved, as we adapt the proof of \cite[Theorem 2.1]{Abou:03} for the unconstrained case to a constrained setting. We compute the term $V_{k+1}^i - V_k^i$ as follows:
		\begin{align*}
			&V_{k+1}^i-V_k^i=-\tfrac{1}{2}\big(\normx{x_k}_{Q^i_k}+2{p_k^{i}}^{\prime}x_k+\sum_{j\in\N}\normx{u_k^{j}}_{R^{i j}_k}\big)+\tfrac{1}{2} {u_k^{i}}^{\prime}Y_k^iu_k^i \\
			&\quad+{u_k^{i}}^{\prime}\big({B_k^{i}}^{\prime}(E_{k+1}^i (A_k x_k+\eta_k^i)+e_{k+1}^i\big)+\tfrac{1}{2}x_k^{\prime}\big(A_k^{\prime}E_{k+1}^i A_k\\
			&\quad-E_k^i+Q^i_k\big) x_k+\big(A_k^{\prime}E_{k+1}^i\eta_k^i+A_k^{\prime}e_{k+1}^{i}+p_k^i-e_k^{i}\big)^{\prime}x_k\\
			&\quad+\tfrac{1}{2}{\eta_k^{i}}^{\prime}(E_{k+1}^i\eta_k^i+2e_{k+1}^i)+\tfrac{1}{2}\sum_{j\in -i}\normx{u_k^{j}}_{R^{i j}_k}+f_{k+1}^i-f_k^i.
		\end{align*}
		Then, we add and subtract the term ${\mu_{k}^{i\star}}^\prime \big(M_k^ix_k+[N_k^i]_i u_k^i+\alpha_k^i+r_k^i\big)$ on the right-hand-side of the above expression and rearrange a few terms in  the above expression as follows
		\begin{align*} 
			&V_{k+1}^i-V_k^i  =-\tfrac{1}{2}\big(\normx{x_k}_{Q^i_k}+2{p_k^i}^\prime x_k+\sum_{j\in\N}\normx{u_k^{j}}_{R^{i j}_k}\big)+\tfrac{1}{2} {u_k^{i}}^{\prime}Y_k^iu_k^i\\
			&\quad+{u_k^{i}}^{\prime}\big({B_k^{i}}^{\prime}(E_{k+1}^i (A_k x_k+\eta_k^i)+e_{k+1}^i {-[N^{i}_k]_i^\prime \mu_k^{i\star}}\big)+\tfrac{1}{2}x_k^{\prime}\big(A_k^{\prime}E_{k+1}^i A_k \\
			&\quad+Q^i_k-E_k^i\big) x_k+\big(A_k^{\prime}E_{k+1}^i\eta_k^i+A_k^{\prime}e_{k+1}^{i}+p_k^i-{M_k^{i}}^\prime \mu_k^{i\star}-e_k^{i}\big)^{\prime}x_k\\
			&\quad+f_{k+1}^i+\tfrac{1}{2}{\eta_k^{i}}^{\prime}(E_{k+1}^i\eta_k^i+2e_{k+1}^i)+\tfrac{1}{2}\sum_{j\in -i}\normx{u_k^{j}}_{R^{i j}_k} \\
			&\quad-{\mu_{k}^{i\star}}^\prime \big(\alpha_k^{i}+r_k^i\big)-f_k^i+{{\mu_{k}^{i\star}}^\prime \big(M_k^ix_k+[N_k^i]_i u_k^i+\alpha_k^i+r_k^i\big)}.
		\end{align*}
		Next, we perform completion of squares in the above expression and with a few algebraic calculations we get
		\begin{align*}
			&V_{k+1}^i-V_k^i
			=-\tfrac{1}{2}\big(\normx{x_k}_{Q^i_k}+2{p_k^{i}}^{\prime}x_k+\sum_{j\in\N}\normx{u_k^{j}}_{R^{i j}_k}\big)+\tfrac{1}{2}\normx{u_k^{i}+y_k^i}_{Y_k^i} \\
			&\quad+{\mu_{k}^{i\star}}^\prime \big(M_k^ix_k+[N_k^i]_i u_k^i+\alpha_k^i+r_k^i\big)
            \end{align*}\begin{align*}
			&\quad+\tfrac{1}{2}x_k^{\prime}\big(A_k^{\prime}E_{k+1}^i A_k+Q^i_k-A_k^{\prime} E_{k+1}^i
			B_k^i \big(Y_k^i\big)^{-1}{B_k^{i}}^{\prime} E_{k+1}^i A_k-E_k^i\big) x_k\\
			&\quad+\big(A_k^{\prime}e_{k+1}^{i}+A_k^{\prime}E_{k+1}^i\eta_k^i-A_k^\prime E_{k+1}^i B_k^{i} b_k^i+p_k^i-{M_k^{i}}^\prime \mu_k^{i\star}-e_k^{i}\big)^{\prime}x_k\\
			&\quad+f_{k+1}^i+\tfrac{1}{2}{\eta_k^{i}}^{\prime}(E_{k+1}^i\eta_k^i+2e_{k+1}^i)\\
			&\quad+\tfrac{1}{2}\sum_{j\in -i}\normx{u_k^{j}}_{R^{i j}_k}-\tfrac{1}{2}\normx{b_k^{i}}_{Y_k^i}-{\mu_{k}^{i\star}}^\prime \big(\alpha_k^{i}+r_k^i\big)-f_k^i.
		\end{align*}
      Using \eqref{eq:EEQ}-\eqref{eq:RiccatiBackward} in the above equation, we get
		\begin{align*}
			V_{k+1}^i-V_k^i
			&=-\tfrac{1}{2}\big(\normx{x_k}_{Q^i_k}+2{p_k^{i}}^{\prime}x_k+\sum_{j\in\N}\normx{u_k^{j}}_{R^{i j}_k}\big)+\tfrac{1}{2}\normx{u_k^{i}+y_k^i}_{Y_k^i} \\
			&\quad+{\mu_{k}^{i\star}}^\prime \big(M_k^ix_k+[N_k^i]_i u_k^i+\alpha_k^i+r_k^i\big).
		\end{align*}
		Taking the telescopic sum of $V_{k+1}^i-V_k^i$ for all $k\in \K_l$, we obtain
		the relation \eqref{eq:CostasV}.
\end{proof}
%
\begin{remark}
Although the objective function \eqref{eq:CostasV} appears to change with the constraints and $\mu_k^{i\star}$ due to the third term, the first term $V_0^i$ in \eqref{eq:CostasV}, defined by the recursive difference equations \eqref{eq:RiccatiBackward}, includes all the same terms with opposite signs. Consequently, the changes in the third term are canceled out. 
\end{remark}
\begin{remark} 
	Theorem \ref{thm:SufficentRiccati} extends \cite[Theorem 2.1]{Abou:03} to settings where players' strategy sets are interdependent, as characterized by the constraints \eqref{eq:constraints}. The result in \cite[Theorem 2.1]{Abou:03} can be recovered by setting \(\mu_k^{i\star} = 0\) for all \(k \in \K_l\) and \(i \in \N\) in Theorem \ref{thm:SufficentRiccati}. 
\end{remark}
\subsubsection{Related static game with coupled constraints} 
Since the information structure is open-loop, Theorem~\ref{thm:SufficentRiccati} allows us to derive from DGC an auxiliary static game with coupled constraints (SGCC), in which Player~$i$'s objective function is given by \eqref{eq:CostasV} and the players' feasible strategy space is defined by \eqref{eq:CConstraint}; see also \cite{Basar:1976, Abou:03} for the unconstrained case. 
Let $(\bar{\bu}^i,\bar{\bu}^{-i})\in R(x_0)$ be a  {generalized} Nash equilibrium associated with SGCC. This implies for each Player $i$ ($i\in \N$),  $\bar{\bu}^i$  solves the following constrained optimization problem
\begin{align}
	J^i(\bar{\bu}^i,\bar{\bu}^{-i})\leq J^i(\bu^i,\bar{\bu}^{-i}),~\forall \bu^i\in \U(\bar{\bu}^{-i}).
	\label{eq:conopt}
\end{align} 
The Lagrangian associated with Player $i$'s problem \eqref{eq:conopt} is given by
\begin{align}
	\bar{\mathsf L}_k^i	&=V_0^i+\tfrac{1}{2}\sum_{k\in\K_l}\normx{u_k^{i}+y_k^i}_{Y_k^i}\nonumber\\ 
	&\quad-\sum_{k\in\K_l}({\mu_k^i}-\mu_{k}^{i\star})^{\prime}(M_k^ix_k+[N_k^i]_i u_k^i+\bar{\alpha}_k^{i}+r_k^i).\label{eq:StaticLagrangian}
\end{align}
Here, the constrained set  $  \U(\bar{\bu}^{-i})$ is   given by $M_k^ix_k+[N_k^i]_i u_k^i+\bar{\alpha}_k^{i}+r_k^i \geq 0$, with $\bar{\alpha}^i_k=\sum_{j \in -i}[N^i_k]_j\bar{u}_k^j$ for all $k\in \K_l$. Using \eqref{eq:state}, the state variables $x_k$ in \eqref{eq:StaticLagrangian}
are written in terms of $\{x_0,~u_l^i,~ l<k,~i\in \N\}$ as follows:
\begin{align*}
	\bar{\mathsf L}_k^i	&=V_0^i+\tfrac{1}{2}\sum_{k\in\K_l}\normx{u_k^{i}+y_k^i}_{Y_k^i}-\sum_{k\in\K_l}({\mu_k^i}-\mu_{k}^{i\star})^{\prime}\Big(M_k^iA_{k-1}\cdots A_{0} x_0\\
	&\hspace{0.5in}+\big(\sum_{\tau=k+2}^{K-1}M_{\tau}^{i}A_{\tau-1}\cdots A_{k+1}B_k^i+M_{k+1}^{i}B_k^i+[N_k^i]_i\big) u_k^i\\
	&\hspace{0.5in}+\big(\sum_{\tau=k+2}^{K-1}M_{\tau}^{i}A_{\tau-1}\cdots A_{k+1}+M_{k+1}^{i}\big) \eta_k^i+\alpha_k^i+r_k^i\Big).
\end{align*}
Then, the  KKT conditions associated with  \eqref{eq:conopt}  are given by
\begin{subequations}\label{eq:staticLCP}
	\begin{align}
		& \bar{u}_k^{i}=-\bar{y}_k^{i}+(Y_k^i)^{-1}\beta_k^i, \label{eq:staticLCP1}\\
		& 0 \leq M_k^i\bar{x}_k+[N_k^i]_i\bar{u}_k^i+\bar{\alpha}_k^{i}+r_k^i\perp \bar{\mu}_k^{i} \geq 0,\label{eq:staticLCP2}
	\end{align}
\end{subequations}
where
\begin{subequations}\label{eq:muequation}
	\begin{align}
		\beta_k^i
		&={[N_k^i]_i}^{\prime}(\bar{\mu}_{k}^i-\mu_{k}^{i\star})+{B_{k+1}^i}^{\prime}{M_{k+1}^{i}}^{\prime}(\bar{\mu}_{k+1}^i-\mu_{k+1}^{i\star})\nonumber\\
		&\quad+\sum_{\tau=k+2}^{K-1}(A_{\tau-1}\cdots A_{k+1}B_k^i)^{\prime}{M_{\tau}^{i}}^{\prime}(\bar{\mu}_{\tau}^i-\mu_{\tau}^{i\star}),\label{eq:betadef}\\
		\bar{y}_k^{i}&=(Y_k^i)^{-1}{B_k^{i}}^{\prime}E_{k+1}^i A_k \bar{x}_k+b_k^i,\label{eq:ybeq}
	\end{align}
\end{subequations}
and the vectors   $\{b_k^i, k\in \K_l\}$ are obtained from \eqref{eq:bEQ}. Here, 
$\{\bar{\mu}_k^i\in\R^{c_i}_{+},~ k\in \K_l\}$ denotes the  set of optimal multipliers associated with the constraint $\U(\bar{\bu}^{-i})$ and  $\{\bar{x}_k, ~k\in \K\}$ denotes the state trajectory generated by the generalized Nash equilibrium strategy $(\bar{\bu}^i,\bar{\bu}^{-i})$.

Next, substituting for the control $\bar{u}_k^i$ from \eqref{eq:staticLCP1} in the state equation
\eqref{eq:state}, and using \eqref{eq:RiccatiBackward} and \eqref{eq:staticLCP}, the KKT conditions associated with minimization problems of all the players are collected
as the following discrete-time coupled linear complementarity system (LCS2) 
\begin{subequations}\label{eq:LCS2}
	\begin{align}
		& \bar{x}_{k+1} =\big(\mathbf{I}-B_k^i\big(Y_k^i\big)^{-1}{B_k^{i}}^\prime E_{k+1}^i\big) A_k \bar{x}_k-B_k^i b_k^i\nonumber\\
		&\hspace{0.75in}+B_k^i(Y_k^i)^{-1}\beta_k^i+\bar{\eta}_k^{i}, \label{eq:xeq2PBCC}\\
		&e_k^{i} =A_k^{\prime}e_{k+1}^{i}+A_k^{\prime}E_{k+1}^i \bar{\eta}_k^{i}-A_k^{\prime}E_{k+1}^i B_k^{i} b_k^i+p_k^i-{M_k^{i}}^\prime\mu_k^{i\star}, \label{eq:eeq2PBCC}\\
		&0 \leq  [N_k^i]_i(Y_k^i)^{-1}\beta_k^i+(M_k^i-[N_k^i]_i (Y_k^i)^{-1}{B_k^{i}}^\prime E_{k+1}^i A_k) \bar{x}_k\nonumber\\
		&\hspace{0.75in}-[N_k^i]_i b_k^i+\bar{\alpha}_k^{i}+r_k^i\perp \bar{\mu}_k^{i} \geq 0,\label{eq:ceq2PBCC}\\
		&Y_k^i b_k^i ={B_k^{i}}^\prime(E_{k+1}^i \bar{\eta}_k^{i}+e_{k+1}^i)-[N^{i}_k]_i^\prime\mu_{k}^{i\star},\label{eq:beq2PBCC}
	\end{align}
\end{subequations}
where  $\bar{\eta}_k^i=\sum_{j\in -i}B^j_k \bar{u}_k^j$ with the boundary conditions are $\bar{x}_0=x_0$, $e_K^i=p_K^i$ and $b_K^i=0$ for $i\in \N$. We denote the solution of LCS2 (with parameters $\bm{\mu}^\star$),  if it exists, by
\begin{align}\{\bm{\bar{x}}, \bm{e} , \bm{\bar{\mu}} ; \bm{b} \}^{2}_{\bm{\mu}^\star}:=\{\bar{x}_k, e_k^i, b_k^i, ~k \in\K,~ \bar{\mu}_k^i,~ k \in\K_l, ~i\in \N \}. \label{eq:LCS2sol}\end{align}  
From \eqref{eq:beq2PBCC}, we note that the variables $\bm{b}$ in \eqref{eq:LCS2sol} are dependent and derived from the independent variables $\{\bm{\bar{x}}, \bm{e} , \bm{\bar{\mu}}\}$. We define a solution of LCS2 with the following property.
\begin{definition}\label{def:consistant}
	A solution $\{\bm{\bar{x}}, \bm{e} , \bm{\bar{\mu}} ; \bm{b} \}^{2}_{\bm{\mu}^\star}$ of LCS2 is referred to as a \emph{consistent} solution of LCS2 if $\bm{\bar{\mu}}=\bm{\mu}^\star$.
\end{definition}
\begin{remark}\label{rem:consistent}
     Following Definition \ref{def:consistant} and \eqref{eq:betadef}, we note that a consistent solution of LCS2 is characterized by $\beta_k^i=0$ for all  $k\in \K_l$ and   $i\in \N$.
\end{remark}
 We note that the optimization problems \eqref{eq:dyncon} and \eqref{eq:conopt} are essentially the same, with the {latter} being a static representation of the former. Consequently,   the necessary conditions of these problems characterized by the solutions of  LCS1 and LCS2 must be related. The next two results establish this relationship. 
\begin{theorem}\label{th:TPBlink} 	
	Let Assumptions \ref{ass:GameAssumption} and \ref{ass:Yassumption} hold.  Let $\{\bm{x}^{\star}, \bm{\lambda}, \bm{\mu}^{\star}\}^{1}$ be a solution of LCS1, and using this,  construct the control sequence as
	\begin{subequations}\label{eq:DefPsiandb}
		\begin{align}
			\bar{u}_k^{i}=-(R^{ii}_k)^{-1}({B^{i}_k}^\prime\lambda_{k+1}^i-[N^{i}_k]_i^{\prime}\mu_{k}^{i\star}),~  k\in \K_l,\label{eq:ThmUdefn}
		\end{align}
		and define the following sequence as
		\begin{align}
			e_k^{i}&= \lambda_k^i-E_k^ix_k^{\star},~k\in \K, \label{eq:Eeq}\\
			b_k^{i}&=\big(Y_k^i\big)^{-1}\big({B_k^{i}}^\prime\big(E_{k+1}^i\sum_{j\in -i}B^j_k \bar{u}_k^{j}+e_{k+1}^{i}\big) \nonumber \\
			&\quad 
			-[N^{i}_k]_i^{\prime}\mu_{k}^{i\star}\big),~k\in \K_l, ~b_K^{i}=0,\label{eq:ThmbDef}
		\end{align}
	\end{subequations}
for all $i\in \N$.
	Then, this derived sequence $\{\bm{x}^{\star}, \bm{e}, \bm{\mu}^{ \star};\bm{b}\}^{2}_{\bm{\mu}^{ \star}}$ provides a consistent solution for LCS2 \eqref{eq:LCS2}.
\end{theorem}
\begin{proof} We show that if we choose $\bm{\bar{x}}=\bm{x}^{\star}$, $\bm{\bar{\mu}}=\bm{\mu}^\star$, and construct $\bm{e}$, $\bm{b}$, and the controls using \eqref{eq:DefPsiandb}, then the derived sequence $\{\bm{x}^{\star}, \bm{e}, \bm{\mu}^{\star}; \bm{b}\}^{2}_{\bm{\mu}^{\star}}$ provides a consistent solution for LCS2.

	From \eqref{eq:ThmUdefn} and ${B_k^{i}}^\prime E_{k+1}^i(A_k x_k^{\star}+\sum_{j\in \N}B^j_k \bar{u}_k^{j}-x_{k+1}^{\star})=0$, we get
	\begin{multline*}
		\big(R^{ii}_k +{B_k^{i}}^\prime E_{k+1}^i B_k^i\big)\bar{u}_k^{i}+{B_k^{i}}^\prime E_{k+1}^i (A_k x_k^{\star}+\sum_{j\in -i}B^j_k \bar{u}_k^{j})\\
		-{B_k^{i}}^\prime E_{k+1}^ix_{k+1}^{\star}+{B_k^{i}}^\prime\lambda_{k+1}^i-
		[N^{i}_k]_i^{\prime}\mu_{k}^{i\star}=0.
	\end{multline*}
	Then, using $e_{k+1}^{i}=\lambda_{k+1}^i-E_{k+1}^ix_{k+1}^{\star}$ and $Y_k^i=R^{ii}_k +{B_k^{i}}^\prime E_{k+1}^i B_k^i$ in the above relation, we obtain 
	$Y_k^i \bar{u}_k^{i}+{B_k^{i}}^\prime E_{k+1}^i A_k x_k^{\star}
	+{B_k^{i}}^\prime\big(E_{k+1}^i\sum_{j\in -i}B^j_k \bar{u}_k^{j}+e_{k+1}^{i}\big)-[N^{i}_k]_i^{\prime}\mu_{k}^{i\star}= 0$, which results in
	\begin{align} 
		\bar{u}_k^{i}&=-(R^{ii}_k)^{-1}({B^{i}_k}^\prime\lambda_{k+1}^i-[N^{i}_k]_i^\prime \mu_{k}^{i\star})\nonumber\\
		&=-\big(Y_k^i\big)^{-1}{B_k^{i}}^\prime E_{k+1}^i A_k x_k^{\star}-b_k^{i}.~ (\text{from}~\eqref{eq:ThmbDef})\label{eq:controlUeqv}
	\end{align} 
	
		Using \eqref{eq:controlUeqv}, we can write \eqref{eq:OPstate} as follows
		\begin{subequations}\label{eq:LCS3}
			\begin{align}
				x_{k+1}^{\star} &=A_k x_k^{\star}-B^i_k\big(R^{ii}_k\big)^{-1}({B^{i}_k}^\prime\lambda_{k+1}^i-[N^{i}_k]_i^{\prime}\mu_{k}^{i\star})+\bar{\eta}_k^{i}\nonumber\\
				&=\big(\mathbf{I}-B_k^i\big(Y_k^i\big)^{-1}{B_k^{i}}^\prime E_{k+1}^i\big)A_k x_k^{\star}-B^i_k b_k^{i}+\bar{\eta}_k^{i},\label{eq:stateeqthm}
			\end{align}
			where $\bar{\eta}_k^{i}=\sum_{j\in -i}B^j_k \bar{u}_k^{j}$, $k\in \K_l$ and  $x_0^{\star}=x_0$.   We note that \eqref{eq:stateeqthm} verifies \eqref{eq:xeq2PBCC} for a consistent solution of LCS2, i.e., with $\beta_k^i=0,~\forall k\in\K_l$ (see Remark \ref{rem:consistent}). 		
			
				Next,  define $\Psi:=A_k^{\prime}e_{k+1}^{i}+A_k^{\prime}E_{k+1}^i \bar{\eta}_k^{i}-A_k^{\prime}E_{k+1}^i B_k^{i} b_k^{i}+p_k^i-{M_k^{i}}^\prime\mu_k^{i\star}$. Adding and subtracting the term $A_k^{\prime}E_{k+1}^i(A_k x_k^{\star}+B_k^i\bar{u}_k^i)$ and noting $\bar{\eta}_k^{i}=\sum_{j\in -i}B^j_k \bar{u}_k^{j}$,  we get 
			
				\begin{align*}
					&\Psi=A_k^{\prime}e_{k+1}^{i}+A_k^{\prime}E_{k+1}^i (A_k x_k^{\star}+\sum_{j\in \N}B^j_k \bar{u}_k^{j})-A_k^{\prime}E_{k+1}^i (A_k x_k^{\star}+B_k^{i} \bar{u}_k^{i})\\
					&\hspace{0.25in}-A_k^{\prime}E_{k+1}^i B_k^{i} b_k^{i}+p_k^i-{M_k^{i}}^\prime\mu_k^{i\star}\\
					&=A_k^{\prime}\lambda_{k+1}^i-A_k^{\prime}E_{k+1}^i(A_k x_k^{\star} 
					- B_k^{i}\big(Y_k^i\big)^{-1}{B_k^{i}}^\prime E_{k+1}^i A_k x_{k}^{\star})+p_k^i-{M_k^{i}}^\prime\mu_k^{i\star}.
				\end{align*}
In the last step above,  we used \eqref{eq:Eeq} for $k+1$ and \eqref{eq:controlUeqv}. Finally, adding and subtracting $Q_k^ix_k^{\star}$ and rearranging the terms we get
				\begin{align*}
					\Psi&=\big(Q_k^ix_k^{\star}+A_k^{\prime}\lambda_{k+1}^i+p_k^i-{M_k^{i}}^\prime\mu_k^{i\star}\big)\\
					&\quad-\big(A_k^{\prime}E_{k+1}^i A_k+Q^i_k-A_k^{\prime}E_{k+1}^i
					B_k^i \big(Y_k^i\big)^{-1}{B_k^{i}}^\prime E_{k+1}^i A_k\big)x_k^{\star}.
			\end{align*}
			Using \eqref{eq:OPlambda} and \eqref{eq:EEQ}, we obtain $\Psi=\lambda_{k}^i-E_{k}^ix_k^{\star}=e_k^{i}$, which implies
			\begin{equation}
				e_k^{i}=A_k^{\prime}e_{k+1}^{i}+A_k^{\prime}E_{k+1}^i \bar{\eta}_k^{i}-A_k^{\prime}E_{k+1}^i B_k^{i} b_k^{i}+p_k^i-{M_k^{i}}^\prime\mu_k^{i\star}, \label{eq:backwdeq}
			\end{equation}
			with the boundary condition $e_K^{i}= \lambda_K^i-E_K^ix_K^{\star}= Q_K^ix_K^{\star}+p_K^{i}-Q^i_K x_K^{\star}= p_K^{i}$. So, \eqref{eq:backwdeq} verifies the backward equation \eqref{eq:eeq2PBCC}.	 
			
			Using $\bar{\alpha}^i_k=\sum_{j \in -i}[N^i_k]_j\bar{u}_k^j$ and \eqref{eq:controlUeqv}, \eqref{eq:OpconstraintComp} can be written as
			\begin{align}
				&0 \leq M_k^ix_k^{\star}-[N_k^i]_i\big(Y_k^i\big)^{-1}{B_k^{i}}^\prime E_{k+1}^i A_k x_{k}^{\star}\notag\\ &\qquad -[N_k^i]_ib_k^{i}+\bar{\alpha}_k^{i}+r_k^i\perp \mu_k^{i\star} \geq 0,\label{eq:CCeqv}
			\end{align}
			and this verifies  \eqref{eq:ceq2PBCC} for a consistent solution of LCS2.
			Finally, using $\bar{\eta}_k^{i}=\sum_{j\in -i}B^j_k \bar{u}_k^{j}$, \eqref{eq:ThmbDef}, is written as
			\begin{align}
				Y_k^i b_k^{i} ={B_k^{i}}^\prime(E_{k+1}^i \bar{\eta}_k^{i}+e_{k+1}^{i}\big)-[N^{i}_k]_i^\prime\mu_{k}^{i\star},
			\end{align}
		\end{subequations}
		which verifies \eqref{eq:beq2PBCC}. So, along with the  boundary conditions $x_0^{\star}=x_0$ and $e_K^{i}= p_K^{i}$, the equations \eqref{eq:LCS3} are exactly same as those which characterize LCS2  given by \eqref{eq:LCS2}.  This implies, that  the derived sequence $\{\bm{x}^{\star}, \bm{e}, \bm{\mu}^{\star}; \bm{b}\}^{2}_{\bm{\mu}^{\star}}$ is  a consistent solution of LCS2.  
	\end{proof}
 
The next result provides a converse to Theorem \ref{th:TPBlink}.
\begin{theorem}\label{th:TPBlink2} 	
	Let Assumptions \ref{ass:GameAssumption} and \ref{ass:Yassumption} hold. Let  $\{\bm{\bar{x}}, \bm{e}, \bm{\bar{\mu}}; \bm{b}\}^{2}_{\bm{\bar{\mu}}}$ be a consistent  solution of \textrm{LCS2}, and using this, for all $i\in \N$, construct the control sequence as
	\begin{subequations}\label{eq:DefLambda}
		\begin{align}
			u_k^{i\star}=-\big(Y_k^i\big)^{-1}{B_k^{i}}^\prime E_{k+1}^i A_k \bar{x}_k-b_k^{i},~  k\in \K_l,\label{eq:ConvThmUdefn}
		\end{align}
		and the co-state sequence  as
		\begin{align}
			\lambda_k^i:=E_k^i\bar{x}_k+e_k^i,~ k\in\K.\label{eq:Lambdaeq}
		\end{align}
	\end{subequations}
      Then, this derived sequence $\{\bm{\bar{x}}, \bm{\lambda}, \bm{\bar{\mu}}\}^{1}$ is a solution for LCS1 \eqref{eq:LQKKT}.
\end{theorem}
\begin{proof} 
	 We show that  if we choose $\bm{x}^{\star}=\bm{\bar{x}}$, $\bm{\mu}^\star=\bm{\bar{\mu}}$ and construct $\bm{\lambda}$ and the controls using \eqref{eq:DefLambda}, then the derived sequence $\{\bm{\bar{x}}, \bm{\lambda}, \bm{\bar{\mu}}\}^{1}$ provides a solution for LCS1, that is, they satisfy \eqref{eq:LQKKT}.
	
 Using definition of $b_k^i$ from \eqref{eq:beq2PBCC} with $\eta_k^{i\star}=\sum_{j\in -i}B^j_k u_k^{j\star}$, \eqref{eq:ConvThmUdefn} is written as $
		Y_k^i u_k^{i\star}+{B_k^{i}}^\prime E_{k+1}^i A_k \bar{x}_k+{B_k^{i}}^\prime\big(E_{k+1}^i\eta_k^{i\star}+e_{k+1}^{i}\big)-[N^{i}_k]_i^\prime \bar{\mu}_{k}^{i}= 0$.
	Next, using  $Y_k^i=R^{ii}_k +{B_k^{i}}^\prime E_{k+1}^i B_k^i$,    $e_{k+1}^i=\lambda_{k+1}^i-E_{k+1}^i\bar{x}_{k+1}$, and writing the state equation \eqref{eq:state} as $A_k \bar{x}_k+\sum_{j\in \N}B^j_k u_k^{j\star}-\bar{x}_{k+1}=0$, the previous equation simplifies to 
	\begin{align}
		u_k^{i\star}&=-\big(Y_k^i\big)^{-1}{B_k^{i}}^\prime E_{k+1}^i A_k \bar{x}_k-b_k^{i},~ (\text{from } \eqref{eq:ConvThmUdefn})\nonumber\\
		&=-(R^{ii}_k)^{-1}({B^{i}_k}^\prime\lambda_{k+1}^i-[N^{i}_k]_i^\prime \bar{\mu}_{k}^{i}).\label{eq:controlUeqv2}
	\end{align}
From Remark \ref{rem:consistent}, for a consistent solution of LCS2, we have $\beta_k^i=0$ for all  $k\in \K_l$ and $i\in \N$. Then, using \eqref{eq:controlUeqv2} in  \eqref{eq:xeq2PBCC}, we get
	\begin{subequations}\label{eq:LCS6}
		\begin{align}
			\bar{x}_{k+1} &=\big(\mathbf{I}-B_k^i\big(Y_k^i\big)^{-1}{B_k^{i}}^\prime E_{k+1}^i\big) A_k \bar{x}_k-B_k^i b_k^i+\sum_{j\in -i}B^j_k u_k^{j\star}\nonumber\\
			&=A_k \bar{x}_k-\sum_{j\in\N}{B}^{j}_k (R_k^{jj})^{-1}\big({B^{j}_k}^\prime\lambda_{k+1}^j-[N^{j}_k]_j^\prime \bar{\mu}_{k}^{j}\big),
		\end{align}
		with initial condition $\bar{x}_0=x_0$, and this verifies \eqref{eq:OPstate}. Next, define $\Phi=Q_k^i\bar{x}_k+p_k^{i}+A^\prime_k\lambda_{k+1}^{i}-{M^{i}_k}^\prime \bar{\mu}_{k}^{i}$. Using \eqref{eq:Lambdaeq}, we get 
			\begin{align*}
				\Phi&=Q_k^i\bar{x}_k+p_k^{i}+A^\prime_k(E_{k+1}^i\bar{x}_{k+1}+e_{k+1}^i)-{M^{i}_k}^\prime \bar{\mu}_{k}^{i}\\
				&=Q_k^i\bar{x}_k+A^\prime_k(E_{k+1}^i(\bar{x}_{k+1}-(A_k \bar{x}_k+B_k^iu_k^{i\star}))+e_{k+1}^i)
				+p_k^{i}-{M^{i}_k}^\prime \bar{\mu}_{k}^{i}\\
				&\quad+A_k^{\prime}E_{k+1}^i(A_k \bar{x}_k+B_k^iu_k^{i\star})\\
				&=Q_k^i\bar{x}_k+A^\prime_k(E_{k+1}^i\eta_k^{i\star}+e_{k+1}^i)
				+p_k^{i}-{M^{i}_k}^\prime \bar{\mu}_{k}^{i}\\
				&\quad+A_k^{\prime}E_{k+1}^i(A_k \bar{x}_k-B_k^i(\big(Y_k^i\big)^{-1}{B_k^{i}}^\prime E_{k+1}^i A_k \bar{x}_k-b_k^{i})).
			\end{align*}
			In the last step, we used $\bar{x}_{k+1}-(A_k \bar{x}_k+B_k^iu_k^{i\star})=\sum_{j\in-i}B_k^ju_k^{j\star}=\eta_k^{i\star}$ and \eqref{eq:controlUeqv2}. Next, rearranging the above terms, we obtain
			\begin{align*}
				\Phi&=\big(A_k^{\prime}E_{k+1}^i A_k+Q^i_k-A_k^{\prime}E_{k+1}^i
				B_k^i \big(Y_k^i\big)^{-1}{B_k^{i}}^\prime E_{k+1}^i A_k\big)\bar{x}_k
				\\
				&\quad+(A_k^{\prime}e_{k+1}^{i}+A_k^{\prime}E_{k+1}^i \eta_k^{i\star}-A_k^{\prime}E_{k+1}^i B_k^{i} b_k^{i}+p_k^i-{M_k^{i}}^\prime \bar{\mu}_{k}^{i}).
		\end{align*}	
 Using \eqref{eq:EEQ} and \eqref{eq:eeq2PBCC}, we obtain $\Phi=E_k^i\bar{x}_k+e_k^i=\lambda_{k}^i$, which implies
		\begin{align}
			\lambda_k^i &= Q_k^i\bar{x}_k+p_k^{i}+A^\prime_k\lambda_{k+1}^{i}-{M^{i}_k}^\prime \bar{\mu}_{k}^{i},
		\end{align}
		with the boundary condition $\lambda_K^i=E_K^i\bar{x}_K+e_K^{i}= Q_K^i\bar{x}_K+p_K^{i}$, which verifies \eqref{eq:OPlambda}. Finally, using $\alpha^{i\star}_k=\sum_{j \in -i}[N^i_k]_ju_k^{i\star}$, $\bm{\bar{\mu}}=\bm{\mu}^{\star}$, $\beta_k^i=0$ and \eqref{eq:controlUeqv2}, we can write   \eqref{eq:ceq2PBCC} as follows
		\begin{align}
			0 \leq M_k^i\bar{x}_k-\sum_{j\in\N} [N^j_k]_j (R_k^{jj})^{-1}\big({B^{j}_k}^\prime\lambda_{k+1}^j-[N^{j}_k]_j^\prime\bar{\mu}_{k}^{j} \big)   +r_k^i\perp \bar{\mu}_{k}^{i}\geq 0,
		\end{align}
	\end{subequations}
which verifies \eqref{eq:OpconstraintComp}. 
So, along with the  boundary conditions $\bar{x}_0=x_0$ and $\lambda_K^i= Q_K^i\bar{x}_K+p_K^{i}$, the equations \eqref{eq:LCS6} are exactly same as those which characterize LCS1  given by \eqref{eq:LQKKT}.  This implies, that  the derived sequence $\{\bm{\bar{x}}, \bm{\lambda}, \bm{\bar{\mu}}\}^{1}$ is  a   solution of LCS1.
\end{proof}
 The next result establishes a relation between the controls synthesized from the solutions of LCS1 and LCS2. 
\begin{corollary}\label{cor:uequiv} 
Let Assumptions \ref{ass:GameAssumption} and \ref{ass:Yassumption} hold. For a solution of LCS1, the control action defined by \eqref{eq:OLNE} is the same as the control action \eqref{eq:staticLCP1} obtained using the derived consistent solution of LCS2 (and vice versa).
	
\end{corollary} 
\begin{proof} Let $\{\bm{x}^{\star}, \bm{\lambda}, \bm{\mu}^{\star}\}^{1}$ be a solution of LCS1. Then 
	from \eqref{eq:OLNE} 
	$u_k^{i\star}:=-(R_k^{ii})^{-1}\big({B^{i}_k}^\prime\lambda_{k+1}^i-[N^{i}_k]_i^\prime \mu_{k}^{i\star}\big)$ is the candidate GOLNE strategy synthesized using this solution.
	Recall, from Theorem \ref{th:TPBlink},  $\{\bm{x}^{\star}, \bm{\lambda}, \bm{\mu}^{\star}\}^{1}$   provides a consistent solution $\{\bm{x}^{\star}, \bm{e}^{\star} , \bm{\mu}^{\star}; \bm{b}^{\star}\}^{2}_{\bm{\mu}^{ \star}}$ for LCS2. Further, from  \eqref{eq:controlUeqv} in the proof of Theorem \ref{th:TPBlink} we have 
	\begin{align}
		u_k^{i\star}&=-(R^{ii}_k)^{-1}({B^{i}_k}^\prime\lambda_{k+1}^i-[N^{i}_k]_i^\prime \mu_{k}^{i\star})\nonumber\\
		&=-\big(Y_k^i\big)^{-1}{B_k^{i}}^\prime E_{k+1}^i A_k x_k^{\star}-b_k^{i\star}=\bar{u}_k^i,\label{eq:CorUeqv} 
	\end{align}
	which is exactly the candidate equilibrium strategy $\bar{u}_k^i$, given by \eqref{eq:staticLCP1} and \eqref{eq:ybeq},
	   synthesized using the (derived) consistent solution $\{\bm{x}^{\star}, \bm{e}^{\star}, \bm{\mu}^{\star}; \bm{b}^{\star}\}^{2}_{\bm{\mu}^{\star}}$ of LCS2; recall also Remark \ref{rem:consistent}. The proof in the other direction follows with similar arguments as above, and from \eqref{eq:controlUeqv2} in the proof of Theorem \ref{th:TPBlink2}.
\end{proof}

\subsubsection{Main result}
Using the equivalence between the necessary conditions \eqref{eq:LQKKT} and \eqref{eq:LCS2}, we are now ready to state the sufficient conditions under which  a solution of LCS1 is  a GOLNE. The next lemma uses an existence result from   static games with coupled constraints \cite[Theorem 1]{Rosen:65} adapted to SGCC.
\begin{lemma} \label{eq:rosenconvex}
	Let Assumptions \ref{ass:GameAssumption}  and \ref{ass:Yassumption} hold. Let the matrices $\{Y_k^i,~k\in \K_l,~i\in \N\}$, as defined in Assumption  \ref{ass:Yassumption}, be positive-definite. Then, SGCC is a convex game and a GOLNE exists for SGCC.
\end{lemma}
\begin{proof} 
	Under Assumption \ref{ass:GameAssumption}.\ref{item:1} the feasible strategy space $R(x_0)$ is non-empty, convex, closed and bounded. Due to positive-definiteness of the matrices $\{Y_k^i,~k\in \K_l\}$, Player $i$'s objective function \eqref{eq:CostasV} is a strictly convex function in $\bu^i$. Then, from \cite[Theorem 1]{Rosen:65}, SGCC  is a convex game, and  consequently, a GOLNE exists for SGCC, which is synthesized  from a solution of LCS2 using \eqref{eq:staticLCP1} and \eqref{eq:ybeq}.
\end{proof} 
\begin{theorem}\label{thm:Nashexistence}
	Let Assumptions \ref{ass:GameAssumption}  and \ref{ass:Yassumption} hold. Let the matrices $\{Y_k^i,k\in \K_l,i\in \N\}$ be positive-definite. Let  $\{\bm{x}^{\star}, \bm{\lambda}, \bm{\mu}^{\star}\}^{1}$ be a solution of \textrm{LCS1}, then $\{u_k^{i\star},k\in\K_l,i\in \N\}$, given by \eqref{eq:OLNE} is a GOLNE for DGC.
\end{theorem}
\begin{proof}  
	For any Player $i \in \N$, we fix the other players' strategies at $\bu^{-i\star}$. For any $\bu^i \in \U(\bu^{-i\star})$, let $\{x_k,~k \in \K\}$ be the generated state trajectory. From Theorem \ref{thm:SufficentRiccati}, Player $i$'s cost \eqref{eq:objective1} can be rewritten as:
	\begin{align}
		J^i(\bu^i, \bu^{-i\star}) &= V_0^{i\star} + \tfrac{1}{2} \sum_{k \in \K_l} \normx{u_k^i + y_k^i}_{Y_k^i} \nonumber \\
		&\quad + \sum_{k \in \K_l} {\mu_k^{i\star}}^\prime \big(M_k^i x_k + [N_k^i]_i u_k^i + {\alpha}_k^i + r_k^i\big), \label{eq:OLNash1}
	\end{align}
	where $\alpha_k^i = \sum_{j \in i^-} [N_k^i]_j u_k^{j\star}$. Furthermore, $V_0^{i\star}$ is obtained using \eqref{eq:RiccatiBackward} with $\eta_k^i = \sum_{j \in i^-} B_k^j u_k^{j\star}$ and is independent of Player $i$'s strategy $\bu^i$. From Theorem \ref{th:TPBlink}, a solution $\{\bm{x}^{\star}, \bm{\lambda}, \bm{\mu}^{\star}\}^1$ of LCS1 provides a consistent solution $\{\bm{x}^{\star}, \bm{e}^{\star}, \bm{\mu}^{\star}; \bm{b}^{\star}\}^2_{\bm{\mu}^{\star}}$ for LCS2. Lemma \ref{eq:rosenconvex} confirms that a GOLNE exists for SGCC. We next show that the controls synthesized from $\{\bm{x}^{\star}, \bm{e}^{\star}, \bm{\mu}^{\star}; \bm{b}^{\star}\}^2_{\bm{\mu}^{\star}}$ (using \eqref{eq:staticLCP1} and \eqref{eq:ybeq}) are a GOLNE for SGCC, and thus for DGC. Setting $\bu^i = \bu^{i\star}$ in \eqref{eq:OLNash1} with the corresponding state trajectory $\{x_k^\star, ~k \in \K\}$, and using Corollary \ref{cor:uequiv}, we have $\bar{u}_k^j = u_k^{j\star}$ for all $j \in \N$. Then, from \eqref{eq:CorUeqv} and the definition of $y_k^i$ (see Theorem 1), the second term on the right-hand side of \eqref{eq:OLNash1} vanishes. Further, the third term also  vanishes due to the consistency of $\{\bm{x}^{\star}, \bm{e}^{\star}, \bm{\mu}^{\star}; \bm{b}^{\star}\}^2_{\bm{\mu}^{\star}}$ in \eqref{eq:ceq2PBCC}; see Remark \ref{rem:consistent}. Thus, we have:
	\begin{align}
		J^i(\bu^{i\star}, \bu^{-i\star}) = V_0^{i\star}. \label{eq:OLNash2}
	\end{align}
	Next, we compare the costs in \eqref{eq:OLNash1} and \eqref{eq:OLNash2} for all $\bu^i \in \U(\bu^{-i\star})$. Since $\mu_k^{i\star} \geq 0$ and $M_k^i x_k + [N_k^i]_i u_k^i + \bar{\alpha}_k^i + r_k^i \geq 0$ for all $\bu^i \in \U(\bu^{-i\star})$, and due to the positive definiteness of $\{Y_k^i,~k \in \K_l\}$, the second and third terms on the right-hand side of \eqref{eq:OLNash1} are non-negative, implying:
	\begin{align*}
		J^i(\bu^{i\star}, \bu^{-i\star}) \leq J^i(\bu^i, \bu^{-i\star}),~\forall \bu^i \in \U(\bu^{-i\star}).
	\end{align*}
	As the choice of Player $i$ is arbitrary, this condition holds for each player $i \in \N$. Therefore, the strategy profile $\{u_k^{i\star},~k \in \K_l,~i \in \N\}$ is a GOLNE for SGCC. Moreover, from Definition \ref{def:OCNEdef}, it is also a GOLNE for DGC.	 
\end{proof}
\begin{figure}[h]
	\centering 
	\begin{tikzpicture}[font=\small,  scale=0.4, 
	squarenode/.style={rectangle, draw=none, fill=YellowOrange!10,  inner sep=3.5pt},
    squarenode1/.style={rectangle, draw=none, fill=white,  inner sep=3pt},]
	\def\ysh{35}
	\def\xsh{60}
	\node[squarenode](DGC)[label={[Blue]}]{DGC};
    \node[squarenode](LCS1)[below of=DGC,xshift=0*\xsh,yshift=-.01*\ysh,label={[Blue]}][label={[Blue]}]{LCS1};
    \node[squarenode](EG)[right of=DGC,xshift=1.25*\xsh,yshift=0,label={[Blue]}]{SGCC};
    \node[squarenode](LCS2)[below of=EG,xshift=0*\xsh,yshift=-.01*\ysh,label={[Blue]}][label={[Blue]}]{LCS2};
    \node[squarenode1](EQ1)[below of=LCS1, yshift=0.35*\ysh]{{$\{\boldsymbol{x}^{\dagger}, \boldsymbol{\lambda}, \boldsymbol{\mu}^{\dagger} \}^{1}$}};
    \node[squarenode1](EQ2)[below of=LCS2, yshift=0.35*\ysh]{{$\{\boldsymbol{x}^{\dagger}, \boldsymbol{e}, \boldsymbol{\mu}^{\dagger};\boldsymbol{b} \}^{2}_{\boldsymbol{\mu}^{\dagger}}$}};
    
    \node[squarenode1](SC)[above of=EQ2, xshift=1.35*\xsh,yshift=0.015*\ysh]{{$\{Y_k^i\succ 0,~k\in\K_l,~i\in\N\}$}};
    %
    \draw[thick, -Stealth] (DGC) -- (EG) node[above, pos=0.5] {Theorem  \ref{thm:SufficentRiccati}};
    \draw[thick, -Stealth] (DGC) -- (LCS1) node[left, pos=0.5]{N.C.};
    \draw[thick, -Stealth] (EG) -- (LCS2) node[right, pos=0.5]{N.C.};
    \draw[thick, -Stealth] ($(EQ1.east)+(0,.2)$) -- ($(EQ2.west)+(-0.1,.2)$) node[above, pos=0.5]{Theorem \ref{th:TPBlink}};
    \draw[thick, Stealth-] ($(EQ1.east)+(0.1,-.2)$) -- ($(EQ2.west)+(0,-.2)$) node[below, pos=0.5]{Theorem  \ref{th:TPBlink2}};
    \draw[thick, Stealth-] (EQ2) -| node {} (SC)
    node[above, pos=0.25]{Theorem \ref{thm:Nashexistence}}
    node[below, pos=0.25]{S.C.};
\end{tikzpicture} 
	\caption{N.C. (S.C.) refer to necessary (sufficient) conditions. Replace $\dagger$ with $\star$ for Theorem \ref{th:TPBlink} and $\dagger$ with $-$ for Theorem \ref{th:TPBlink2}.}
	\label{fig:dependency}
\end{figure}
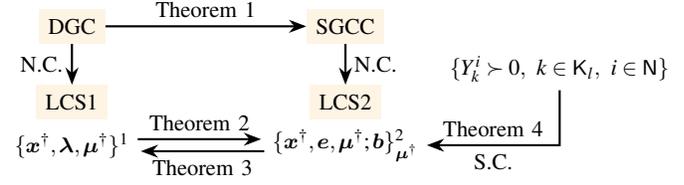
Fig. \ref{fig:dependency} illustrates the dependency between the results obtained so far. The following theorem summarizes both the necessary and sufficient conditions for the existence of a GOLNE. 
\begin{theorem}\label{th:SuffnadUniquness}
	Let Assumption \ref{ass:GameAssumption}  and \ref{ass:Yassumption} hold. Let the matrices $\{Y_k^i,~k\in \K_l,~i\in\N\}$ be positive definite. Then,
\begin{enumerate}[label=(\roman*)]
		\item A GOLNE exists if and only if LCS1 \eqref{eq:LQKKT} has a solution.
		\item GOLNE is unique if and only if LCS1 \eqref{eq:LQKKT} is uniquely solvable.
		\item $\{u_k^{i\star}, k\in\K_l, i\in \N\}$ given by \eqref{eq:OLNE} using the solution of LCS1 \eqref{eq:LQKKT} is a GOLNE. Further, 
     the GOLNE cost of Player $i$ ($i \in\N$) is given by
		\begin{align}
			&J^i (\bu^{i\star}, \bu^{-i\star}) =\tfrac{1}{2}\normx{x_0}_{E_0^i}+{e_0^{i}}^\prime x_0+f_0^i.\label{eq:NashCost}
		\end{align} 
		\end{enumerate}
\end{theorem}
\begin{proof} 
\begin{enumerate}[label=(\roman*)]
	\item From Assumption \ref{ass:GameAssumption}, if a GOLNE exists, then LCS1, being a necessary condition, must be solvable. Conversely, under Assumption \ref{ass:Yassumption} and positive definiteness of the matrices $\{Y_k^i,~k \in \K_l,~i \in \N\}$, when LCS1 has a solution, then from Theorem \ref{thm:Nashexistence},   $\{u_k^{i\star}, k \in \K_l, i \in \N\}$, given by \eqref{eq:OLNE}, is a GOLNE.
	\item From Theorem \ref{thm:Nashexistence}, each solution of LCS1 provides a GOLNE. Thus, if the GOLNE is unique, then LCS1 cannot have more than one solution. Conversely, when LCS1 is uniquely solvable, then due to  Assumption \ref{ass:GameAssumption}.\ref{item:3} (positive definiteness of $R_k^{ii}$), and from Remark \ref{rem:NashNecessay}, there cannot exist more than one GOLNE.
	\item From Theorem \ref{thm:Nashexistence}, the GOLNE is given by \eqref{eq:OLNE} using the solution of LCS1 \eqref{eq:LQKKT}. The GOLNE cost of Player $i$ follows from \eqref{eq:OLNash2}.
\end{enumerate}
\end{proof}
\begin{remark}\label{rem:invertibility} 
%
	Theorem \ref{th:SuffnadUniquness} requires only that the matrices $\{Y_k^i,~ k \in \K_l,~i \in \N\}$ be positive definite, which can be verified a priori using problem data given by \eqref{eq:DGC}, without solving the game. No assumption of positive semi-definiteness is made on $\{Q_k^i,~ k \in \K,~i \in \N\}$. Under Assumption \ref{ass:GameAssumption}.\ref{item:3}, the semi-definiteness of $\{Q_k^i,~ k \in \K,~i \in \N\}$ is sufficient for the positive definiteness of $\{Y_k^i,~k \in \K_l,~i \in \N\}$.		
\end{remark}
\subsection{Reformulation of LCS1 as a linear complementarity 	problem}\label{sec:Solvability}  
In this section, under additional assumptions, we reformulate LCS1 \eqref{eq:LQKKT} as a large-scale linear complementarity problem (LCP) using an approach similar to \cite{Reddy:15}. 
This approach eliminates the state and co-state variables of LCS1, allowing \eqref{eq:LQKKT} to be expressed explicitly in terms of the multipliers ($\bm{\mu}^{\star}$). We now state the following assumption.
\begin{subequations}
		\begin{assumption}\label{ass:LambdaaffineX} The solution $\{P_k^i,~k\in \K \}$ of the following symmetric matrix Riccati difference   equation
			\begin{align} 
				P^i_{k} &= Q_k^i+A^\prime_k P^i_{k+1}\Lambda_{k}^{-1}A_k,~\Lambda_k:=\mathbf{I}+\sum_{j\in \N} B_k^j(R_k^{jj} )^{-1}{B_k^j}^\prime P_{k+1}^j, \label{eq:Pbackward}
			\end{align}
			with $P_K^i=Q_K^i$ exists for all $i\in \N$ (and, thus the matrices $\{\Lambda_k,~k\in \K_l\}$ are invertible). 
		\end{assumption}
		The next proposition follows from \cite[Remark 1.1]{Abou:03}.
		\begin{proposition}	Let Assumptions \ref{ass:GameAssumption}  and \ref{ass:LambdaaffineX} hold. The boundary value problem defined by \eqref{eq:OPstate}-\eqref{eq:OPlambda} has a unique solution $\lambda_k^i:=P_k^ix_k^\star+\zeta_k^i$, where $P_k^i$ satisfies \eqref{eq:Pbackward} and $\zeta_k^i$ satisfies the following backward linear recursive equation for $k\in \K_l$
			\begin{align}
				\zeta_k^i&=A^\prime_k\zeta^i_{k+1}-{M^{i}_k}^\prime\mu_{k}^{i\star}-A^\prime_k P^i_{k+1} \Lambda_{k}^{-1}\nonumber\\
				&\quad\times\sum_{j \in \N}B^j_k (R^{jj}_k)^{-1}({B^j_k}^{\prime}\zeta^j_{k+1}-[N^{j}_k]_j^{\prime}\mu_{k}^{j\star})+p_k^{i},~\zeta_K^i=p_K^i.  \label{eq:Zetabackward}
			\end{align}
		\end{proposition} 
\end{subequations}
Next, under Assumption \ref{ass:LambdaaffineX} and from \eqref{eq:OLNE}, we obtain
\begin{align}
	u_k^{i\star}=-(R^{ii}_k)^{-1}\big({B^{i}_k}^\prime (P^i_{k+1}x_{k+1}^{\star} +\zeta^i_{k+1})-[N^{i}_k]_i^\prime\mu_{k}^{i\star}\big).\label{eq:OPcontrol2}
\end{align}
Upon substituting \eqref{eq:OPcontrol2} in \eqref{eq:OPstate}, we get
	\begin{align}
		x_{k+1}^{\star}= \Lambda_k^{-1}\big(A_k x_k^{\star}-\sum_{j\in\N}{B}^{j}_k (R_k^{jj})^{-1}({B^{j}_k}^\prime\zeta^j_{k+1}-[N^{j}_k]_j^\prime \mu_{k}^{j\star} )\big).\label{eq:stateL1}
	\end{align}
Using the notations defined in the Appendix, and with Assumption \ref{ass:LambdaaffineX}, the joint GOLNE strategy profile \eqref{eq:OPcontrol2} and the vector representation of LCS1 \eqref{eq:LQKKT} are given by
\begin{subequations}\label{eq:LCS5}
	\begin{align}
		&\bu_k^{\star}=\mathbf{F}_{k}x_{k}^{\star}+\overline{\mathbf{F}}_{k+1}\bm{\zeta}_{k+1}+\tilde{\mathbf{F}}_{k}\bm{\upmu}_{k}^{\star},\label{eq:OLNE3}\\
		&x_{k+1}^{\star} =\mathbf{G}_k x_{k}^{\star}+\overline{\mathbf{G}}_{k+1}\bm{\zeta}_{k+1}+\tilde{\mathbf{G}}_k\bm{\upmu}_{k}^{\star},~ x_0^{\star}=x_0,\label{eq:Forwardeq2}\\
		&\bm{\zeta}_k=\mathbf{H}_{k+1}\bm{\zeta}_{k+1}+\overline{\mathbf{H}}_{k}\bm{\upmu}_{k}^{\star}+\mathbf{p}_{k},~\bm{\zeta}_{K} = \mathbf{p}_K,\label{eq:Backwardeq2}\\
		&0 \leq (\overline{\mathbf{M}}_k+\overline{\mathbf{N}}_k\mathbf{F}_{k}) x_k^{\star}+\overline{\mathbf{N}}_k\overline{\mathbf{F}}_{k+1}\bm{\zeta}_{k+1}+\overline{\mathbf{N}}_k\tilde{\mathbf{F}}_{k}\bm{\upmu}_{k}^{\star}+\mathbf{r}_k \perp \bm{\upmu}_{k}^{\star} \geq 0.\label{eq:ComplementCond2}
	\end{align}
\end{subequations}
The linear forward and backward difference equations \eqref{eq:Forwardeq2}-\eqref{eq:Backwardeq2}, can be solved as follows:
	\begin{subequations}\label{eq:LCS61}
		\begin{align}
			x_{k}^{\star} &=\phi(k, 0)x_{0}+\sum_{\tau=1}^{k}\phi(k, \tau)\big(\overline{\mathbf{G}}_{\tau}\bm{\zeta}_{\tau}+ \tilde{\mathbf{G}}_{\tau-1}\bm{\upmu}_{\tau-1}^{\star}\big), \label{eq:StateasZandMU1}\\
			\bm{\zeta}_{k}&=\sum_{\tau=k}^{K-1}\varphi(k, \tau)\overline{\mathbf{H}}_{\tau}\bm{\upmu}_{\tau}^{\star}+\sum_{\tau=k}^{K}\varphi(k, \tau)\mathbf{p}_{\tau},\label{eq:ZetaVectorasMU1}
		\end{align}
	\end{subequations}
		where $\phi(k, \tau)$ and $\varphi(k, \tau)$ are the state transition matrices associated with the forward and backward difference equations  \eqref{eq:Forwardeq2} and \eqref{eq:Backwardeq2}, respectively and are given as
		\[\phi(k, \tau)= 
		\begin{cases}
			\mathbf{G}_{k-1}\mathbf{G}_{k-2}\cdots\mathbf{G}_{\tau}, &  \;\textrm{for}\,\, k<\tau\\
			\mathbf{I},&  \; \textrm{for}\,\, k=\tau \\
		\end{cases} ~~\textrm{with}~~ k, \tau \in \K,
		\]
		\[\varphi(k, \tau)= 
		\begin{cases}
			\mathbf{H}_{k+1}\mathbf{H}_{k+2}\cdots\mathbf{H}_{\tau}, &  \;\textrm{for}\,\, k<\tau\\
			\mathbf{I},&  \; \textrm{for}\,\, k=\tau \\
		\end{cases}  ~~\textrm{with}~~ k, \tau \in \K_l.
		\]
		Now using \eqref{eq:ZetaVectorasMU1}, the equation \eqref{eq:StateasZandMU1} can be written as
		\begin{align}
			x_{k}^{\star} &=\phi(k, 0)x_{0}+\sum_{\tau =1}^{K}\Big(\sum_{\rho=1}^{\mathrm{min}\,\,(k,\tau)}\phi(k, \rho)\overline{\mathbf{G}}_{\rho}\varphi(\rho, \tau)\Big)\mathbf{p}_{\tau}\nonumber\\
			&\quad+\sum_{\tau =1}^{K-1}\Big(\sum_{\rho=1}^{\mathrm{min}\,\,(k,\tau)}\phi(k, \rho)\overline{\mathbf{G}}_{\rho}\varphi(\rho, \tau)\Big)\overline{\mathbf{H}}_{\tau}\bm{\upmu}_{\tau}^{\star}\nonumber\\
			&\quad+\sum_{\tau=1}^{k}\phi(k, \tau)\tilde{\mathbf{G}}_{\tau-1}\bm{\upmu}_{\tau-1}^{\star}.\label{eq:StateasZandMU11}
	\end{align}
Next, using \eqref{eq:ZetaVectorasMU1} in \eqref{eq:StateasZandMU1}, aggregating the variables for all \( k \in \K_l \), and with the notations defined in the Appendix, equations \eqref{eq:OLNE3}, \eqref{eq:LCS61}, and \eqref{eq:ComplementCond2} are compactly written as:
\begin{subequations}\label{eq:OLNEaggrEq}
	\begin{align}
		&\bu_\K^{\star}=\mathbf{F}_\K x_\K^{\star}+\overline{\mathbf{F}}_\K\bm{\zeta}_{\K}+\tilde{\mathbf{F}}_\K\bm{\upmu}_{\K}^{\star},\label{eq:aggregateControl}\\
		& x_{\K}^{\star}=\bm{\Phi}_0x_0+\bm{\Phi}_1\mathbf{p}_{\K}+\bm{\Phi}_2\bm{\upmu}_{\K}^{\star},\label{eq:aggregateState}\\
		&   \bm{\zeta}_{\K}=\bm{\Psi}_1\mathbf{p}_{\K}+\bm{\Psi}_2\bm{\upmu}_{\K}^{\star},\label{eq:aggregateCoState}\\
		&   0 \leq \mathsf{M}_\K x_\K^{\star}+\mathsf{N}_\K\bm{\zeta}_\K+\tilde{\mathsf{N}}_\K\bm{\upmu}_\K^{\star}+\mathbf{r}_\K \perp \bm{\upmu}_\K^{\star} \geq 0.\label{eq:VectorLCP}
	\end{align}
\end{subequations}

The following theorem provides a reformulation of LCS1 \eqref{eq:LQKKT}, which characterizes GOLNE, as a large-scale LCP.
\begin{theorem}\label{th:OLNElcp}
	Let   Assumptions \ref{ass:GameAssumption}, \ref{ass:Yassumption} and \ref{ass:LambdaaffineX} hold true. In addition to this let the matrices $\{Y_k^i=R^{ii}_k+{B_k^{i}}^\prime E_{k+1}^iB_k^i,~k\in\K_l,~i\in\N\}$ be positive definite. Then, the joint GOLNE strategy profile is given by 
	\begin{align}
		\bu_\K^{\star}=\mathsf{F}\bm{\upmu}_{\K}^{\star}+\mathsf{P}(x_0). \label{eq:FinalOLNE}
	\end{align}
	Here, $\bm{\upmu}_{\K}^{\star}$ represents a solution of the following  large-scale linear complementarity problem  
	\begin{align}
		\mathrm{LCP}(x_0):\quad 	0 \leq \mathsf{M}\bm{\upmu}_{\K}^{\star}+\mathsf{q}(x_0) \perp \bm{\upmu}_{\K}^{\star} \geq 0, \label{eq:FinalOLNElcp}
	\end{align}
	where  
$\mathsf{M}:=\mathsf{M}_\K\bm{\Phi}_2+\mathsf{N}_\K\bm{\Psi}_2+\tilde{\mathsf{N}}_\K$, $\mathsf{q}(x_0):=\mathsf{M}_\K \bm{\Phi}_0x_0+(\mathsf{M}_\K\bm{\Phi}_1+\mathsf{N}_\K\bm{\Psi}_1)\mathbf{p}_{\K}+\mathbf{r}_\K$, $\mathsf{F}:=\mathbf{F}_\K\bm{\Phi}_2+\overline{\mathbf{F}}_\K\bm{\Psi}_2+\tilde{\mathbf{F}}_\K$,
	$\mathsf{P}(x_0):=\mathbf{F}_\K \bm{\Phi}_0x_0+(\mathbf{F}_\K\bm{\Phi}_1+\overline{\mathbf{F}}_\K\bm{\Psi}_1)\mathbf{p}_{\K}$.
\end{theorem}
\begin{proof}
	Substituting \eqref{eq:aggregateState} and \eqref{eq:aggregateCoState} into \eqref{eq:aggregateControl} and \eqref{eq:VectorLCP}, respectively, yields the strategy profile \eqref{eq:FinalOLNE} and the \(\mathrm{LCP}\) \eqref{eq:FinalOLNElcp}. Since the matrices \(Y_k^i\) are positive definite for all \(k \in \K_l\) and \(i \in \N\), it follows from Theorem \ref{th:SuffnadUniquness} that the strategy profile given by \eqref{eq:FinalOLNE} is indeed a GOLNE.
\end{proof}
 
\begin{remark}
The non-singularity of the matrices $\{\Lambda_k,~ k \in \K_l\}$, defined in Assumption \ref{ass:LambdaaffineX}, can be verified using the problem data; see also Remark \ref{rem:invertibility}. 
	We observe that \(\mathrm{LCP}(x_0)\) is parametric with respect to \(x_0 \in \mathbb{R}^n\). The set \(\left\{x_0 \in \mathbb{R}^n ~|~ \mathrm{LCP}(x_0) \neq \emptyset\right\}\) represents all initial conditions for which a GOLNE exists for the DGC. If \(\mathrm{LCP}(x_0)\) yields multiple solutions, each is a GOLNE by Theorem \ref{th:OLNElcp}. The existence of LCP solutions and related numerical methods are well-studied in the optimization community; see \cite{Pang:09} for details. 
\end{remark} 
 
\section{Numerical Illustration}
\label{sec:Numerical}
\begin{figure}[H]
	\centering
	\includegraphics[scale=.675]{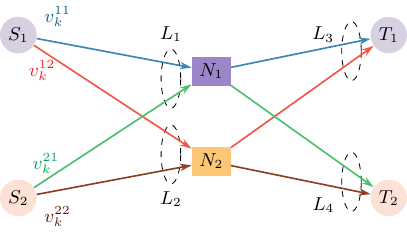}
	\caption{Network flow game with two players and two relay nodes.} 
	\label{fig:relay}
\end{figure}
We illustrate our results using a simplified version of the network-flow control game from \cite{Zazo:16}. The problem involves two sources ($S_1$, $S_2$), two destinations ($T_1$, $T_2$), and two relay nodes ($N_1$, $N_2$) as shown in Fig. \ref{fig:relay}. Each player $i \in \{1, 2\}$ can choose one of two paths ($S_i-N_l-T_i$, where $l \in \{1, 2\}$) to transmit data. The flow for player $i$ through relays $N_1$ and $N_2$ at time $k$ is denoted as $v_k^{i1}$ and $v_k^{i2}$, respectively. Relay nodes have batteries that deplete proportionally to the outgoing flow, described by:
\begin{align}
	x_{k+1}^l=x_k^l-\delta^l(v_{k}^{1l}+v_{k}^{2l}),\label{eq:batterydynamics}
\end{align}
where $\delta^l > 0$ is the depletion factor. Each player $i$ aims to maximize a rate-dependent concave utility function $\sum_{l=1}^{2}(w^{il}v_k^{il} - \frac{1}{2}t^{il}(v_k^{il})^2)$, where $w^{il}$ and $t^{il}$ are positive constants. Additionally, players gain a payoff based on battery levels, modeled by the concave function $\sum_{l=1}^{2}(d^{il}x_k^{l} - \frac{1}{2}s^{il}(x_k^l)^2)$ at decision times $k \in \K_l$, and by $\sum_{l=1}^{2}(D^{il}x_K^{l} - \frac{1}{2}S^{il}(x_K^l)^2)$ at time $K$. Thus, each player $i \in \{1, 2\}$ minimizes the cost function defined as:
\begin{multline}
	J^{i}=-\beta^K\sum_{l=1}^{2}\big(D^{il}x_K^{l}-\tfrac{1}{2}S^{il}(x_{K}^{l})^2\big)\\ -\sum_{k=0}^{K-1}\sum_{l=1}^{2}\beta^k\Big(d^{il}x_k^{l}-\tfrac{1}{2}s^{il}(x_{k}^{l})^2+w^{il}v_k^{il}-\tfrac{1}{2}t^{il}(v_k^{il})^2\Big), \label{eq:numericalobj}
\end{multline}
where $\beta\in(0, 1)$ represents the discount factor. The constraints at each time instant $k\in \K_l$ are given by
\begin{subequations}\label{eq:NFconstraints}
	\begin{align}
		& \text{Relay: } L_k^l=v_{k}^{1l}+v_{k}^{2l} \leq c^l >0,~l=1,2,\label{eq:CoupleConst}\\
		&\text{Destination: }  \tilde{L}_k^i=  v_{k}^{i1}+v_{k}^{i2} \leq \tilde{c}^i>0,~i=1,2,\label{eq:PrivateConst}\\
		&\text{Battery:  }  x_{k+1}^l=  x_k^l-\delta^l(v_{k}^{1l}+v_{k}^{2l})\geq {b^l_{min}} ,~l=1,2,\label{eq:StateCoupleConst}\\
		&\text{Transmission: }  0\leq v_{k}^{il} \leq \bar{v}^{il}= \tfrac{w^{il}}{t^{il}},~i,l=1,2.\label{eq:BoxConst}
	\end{align}
\end{subequations}   
The coupled constraint \eqref{eq:CoupleConst} and individual constraint \eqref{eq:PrivateConst} represent capacity limits for both relay and destination nodes. The mixed constraint \eqref{eq:StateCoupleConst} ensures that a relay node $l \in \{1, 2\}$ cannot transmit data below its minimum battery level, ${b^l_{min}} \geq 0$. Constraint \eqref{eq:BoxConst} further ensures that data transmitted from each source is non-negative and does not exceed $\frac{w^{il}}{t^{il}}$, corresponding to the concave increasing region of the rate-dependent payoff. 
Both players share relay nodes, and their transmission rates are interdependent due to the capacity and battery constraints in \eqref{eq:CoupleConst} and \eqref{eq:StateCoupleConst}. 
The network flow game is represented in standard form \eqref{eq:DGC}, with the state variable $x_k = \col{x_k^1, x_k^2, z_k}$, where $z_k = 1$ for $k \in \K$, and the decision variable for player $i$ on path $l$ is $u_k^{il} = v_k^{il} - \frac{w^{il}}{t^{il}}$ for $i=1, 2$ and $l=1, 2$. The network flow game in standard form \eqref{eq:DGC} is obtained as follows
\begin{align*}
		&    A_k=\begin{bmatrix}
			1            &0   &-\delta^1\alpha^1\\
			0         &1  &-\delta^2\alpha^2\\
			0         &0   &1
		\end{bmatrix},~
		B_k^1=B_k^2=\begin{bmatrix}
			-\delta^1    &0 \\
			0    &-\delta^2 \\
			0     &0    
		\end{bmatrix},\\
		& M^i_k=\col{\mathbf{0}_{6\times 3}, \tilde{M} ,\mathbf{0}_{1\times 3}},~r_k^i=\col{[r^i]_l}_{l=1}^{9},~i=1,2,\\
		&[N^1_k]_2=[N^2_k]_1=\col{\mathbf{0}_{4\times2},\tilde{N}_1,\tilde{N}_2,\mathbf{0}_{1\times 2}},\\
		&[N^1_k]_1=[N^2_k]_2=\col{\mathbf{I}_{2\times2},-\mathbf{I}_{2\times2},\tilde{N}_1,\tilde{N}_2, [-1 ~-1]},\\
		&   \tilde{M}=\begin{bmatrix}
			1        &0    &0   \\
			0         &1 &0  \\
		\end{bmatrix},~ \tilde{N}_1=-\mathbf{I}_{2\times 2},~
		\tilde{N}_2=\begin{bmatrix}
			-\delta^1    &0 \\
			0 &-\delta^2  
		\end{bmatrix}, \\
		& Q^i_k =\beta^k\begin{bmatrix}
			s^i &-d^i\\
			-d^{i\prime} &-\alpha^{i+4}
		\end{bmatrix},~Q^i_K =\beta^K\begin{bmatrix}
			S^i &-D^i\\
			-D^{i\prime} &0
		\end{bmatrix},\\     &R_k^{ii}=\beta^k\oplus_{l=1}^{2}t^{il},~p_k^i=0,~R_k^{ij}=0,~i,j=1,2,~i\neq j,\\
		&  \alpha^1=\tfrac{w^{11}}{t^{11}}+\tfrac{w^{21}}{t^{21}},~\alpha^2=\tfrac{w^{12}}{t^{12}}+\tfrac{w^{22}}{t^{22}},~\alpha^3=\tfrac{w^{11}}{t^{11}}+\tfrac{w^{12}}{t^{12}},\\
		&\alpha^4=\tfrac{w^{21}}{t^{21}}+\tfrac{w^{22}}{t^{22}},~ \alpha^5=\tfrac{(w^{11})^2}{t^{11}}+\tfrac{(w^{12})^2}{t^{12}},~\alpha^6=\tfrac{(w^{21})^2}{t^{21}}+\tfrac{(w^{22})^2}{t^{22}},\\
		&[r^i]_1= \tfrac{w^{i1}}{t^{i1}},~[r^i]_2= \tfrac{w^{i2}}{t^{i2}},~[r^i]_{3}=[r^i]_{4}=0,~[r^i]_5=c^{1}-\alpha^{1},\\
		&[r^i]_{6}=c^{2}-\alpha^{2},~[r^i]_{7}=-\delta^1\alpha^{1}-b,
		[r^i]_{8}=-\delta^2\alpha^{2}-b,\\
		&[r^1]_{9}=\tilde{c}^1-\alpha^3,~[r^2]_{9}=\tilde{c}^2-\alpha^4,~S^i=\oplus_{l=1}^{2}S^{il},~s^i=\oplus_{l=1}^{2}s^{il},\\
		& d^i=\col{d^{il}}_{l=1}^{2},~D^i=\col{D^{il}}_{l=1}^{2},~i=1,2.
\end{align*}
The problem parameters are set as follows: $K=60, ~x_0^1=9,~x_0^2=6,~b^1_{\min}=1,~b_{\min}^2=0.5,~\delta^1=0.125,~\delta^2=0.075,~\beta=0.95,~w^{11}=30,  ~ w^{12}=16,~w^{21}=26,~w^{22}=10,~ t^{i1}=t^{i2}=10,~d^{il}=6,~D^{il}=8.0, ~s^{11}=6,~s^{12}=5.5,~s^{21}=7,~s^{22}=7.5,~S^{11}=8,~S^{12}=8.5,~S^{21}=9,~S^{22}=9.5,~
{c}^1=5.4,~{c}^2=2.6,~\tilde{c}^1=4,~\tilde{c}^2=3$. 
\begin{remark}  \label{rem:potential} 
The open-loop potential game approach of \cite{Zazo:16} is applicable to the above flow control game only when \(s^{1l} = s^{2l}\) and \(S^{1l} = S^{2l}\) for \(l = 1, 2\), as only in this case are the sufficient conditions in \cite[Lemma 3 and Lemma 4]{Zazo:16} satisfied. Nevertheless, these constraints restrict the range of strategic scenarios within the flow-control game. Consequently, the approach presented in \cite{Zazo:16} is unsuitable for computing a GOLNE in such cases. 
\end{remark}
\begin{figure}[t] 
	\centering 
	\pgfplotsset{axis line style={black!25},}
\subfloat[]{\label{fig:fig21}
	\begin{tikzpicture}[             scale=0.5  ,>=latex']
		\begin{axis}[xlabel= Time steps, grid=both,
			grid style={line width=.1pt, draw=gray!10}, minor tick num=3, xmin=0, xmax=60,ymin=-0.15,ymax=3.1]
			\addplot[MidnightBlue, very thick] table [x index=0, y index=1]{NonSymmOpenloop.dat}; 
			\addlegendentry{$v_k^{11}$};
			\addplot[Red,   very thick] table [x index=0, y index=2]{NonSymmOpenloop.dat}; 
			\addlegendentry{$v_k^{12}$};
		 \addplot[Green, very thick ] table [x index=0, y index=3]{NonSymmOpenloop.dat}; 
		\addlegendentry{$v_k^{21}$};
		\addplot[Sepia, very thick ] table [x index=0, y index=4]{NonSymmOpenloop.dat}; 
		\addlegendentry{$v_k^{22}$};	
            \draw[<-, black] (13,1.6) to (14.5, 2)node[right]{\eqref{eq:BoxConst}};
            \draw[<-, black] (9,1) to (6.5, 0.6)node[below]{\eqref{eq:BoxConst}};
   	        \draw[<-, black] (48,0) to (48, 0.3)node[above]{\eqref{eq:BoxConst}};
   	        \draw[-,black!60,dashed](0,3) to (60,3) node[xshift=-3.5cm,below,black]{$\bar{v}^{11}$}; 
   	        \draw[-,black!60,dashed](0,0) to (60,0);
            \draw[-,black!60,dashed](0,1.6) to (60,1.6) node[xshift=-2cm,below,black]{$\bar{v}^{12}$}; 
   	        \draw[-,black!60,dashed](0,2.6) to (60,2.6) node[xshift=-2.75cm,below,black]{$\bar{v}^{21}$};  
   	        \draw[-,black!60,dashed](0,1) to (60,1) node[xshift=-1cm,above,black]{$\bar{v}^{22}$}; 
   	        \draw[-,black!60,dashed](0,0) to (60,0);
		\end{axis}
\end{tikzpicture}}  
\qquad 
\subfloat[]{\label{fig:fig22}
	\begin{tikzpicture}[             scale=0.5  ,>=latex']
		\begin{axis}[xlabel= Time steps, grid=both,
			grid style={line width=.1pt, draw=gray!10}, minor tick num=3, xmin=0, xmax=60,ymin=-0.75,ymax=9.5]
			\addplot[Violet,   very thick] table [x index=0, y index=1]{NonSymmOpenloopstate.dat}; 
			\addlegendentry{$x_k^{1}$};
			\addplot[YellowOrange, very thick] table [x index=0, y index=2]{NonSymmOpenloopstate.dat};
			\addlegendentry{$x_k^{2}$};
			\draw[<-, black] (45,0.5) to (45, -0.25)node[right]{\eqref{eq:StateCoupleConst}};
			\draw[<-, black] (21,1) to (21, 0.2)node[below]{\eqref{eq:StateCoupleConst}};
			\draw[-,black!60,dashed](0,1) to (50,1)node[xshift=-4.5cm,above,black]{$b^1_{\min}$};
			\draw[-,black!60,dashed](0,0.5) to (50,0.5)node[xshift=-4.5cm,below,black]{$b^2_{\min}$};
		\end{axis}
\end{tikzpicture}} \\
\subfloat[]{\label{fig:fig23}
	\begin{tikzpicture}[             scale=0.5  ,>=latex']
		\begin{axis}[xlabel= Time steps, grid=both,
			grid style={line width=.1pt, draw=gray!10}, minor tick num=3, xmin=0, xmax=60, ymin=-0.15,ymax=5.6]
	       \addplot[Violet, very thick] table [x index=0, y index=5]{NonSymmOpenloop.dat}; 
	       \addlegendentry{$L_k^{1}$};
	       \addplot[YellowOrange, very thick] table [x index=0, y index=6]{NonSymmOpenloop.dat};
	       \addlegendentry{$L_k^{2}$};
	        \draw[<-, black] (10,2.6) to (9,2)node[below]{\eqref{eq:CoupleConst}};
	        \draw[-,black!60,dashed](0,5.4) to (60,5.4) node[xshift=-4cm,below,black]{$c^1$};  
	        \draw[-,black!60,dashed](0,2.6) to (60,2.6) node[xshift=-1.5cm,below,black]{$c^2$};  
 	        \draw[<-, black] (20.5,0.2) to (23.5, 0.5)node[right ]{\eqref{eq:StateCoupleConst}};
   	        \draw[<-, black] (44,0.1) to (46.5, 0.5)node[right ]{\eqref{eq:StateCoupleConst}};
	        \addplot[black, dashed ] table [x index=0, y index=9]{NonSymmOpenloop.dat} node[right,pos=0.575]{$\frac{x^1_k-b^1_{\min}}{\delta^1}$};  
	        \addplot[black, dashed ] table [x index=0, y index=10]{NonSymmOpenloop.dat}
	         node[right,pos=0.75]{$\frac{x^2_k-b^2_{\min}}{\delta^2}$}; 
		\end{axis}
\end{tikzpicture}}  
\qquad
\subfloat[]{\label{fig:fig24}
	\begin{tikzpicture}[             scale=0.5  ,>=latex']
		\begin{axis}[xlabel= Time steps, grid=both,
			grid style={line width=.1pt, draw=gray!10}, minor tick num=3, xmin=0, xmax=60, ymin=-0.15,ymax=4.2]
						
			\addplot[blue!50,   very thick] table [x index=0, y index=7]{NonSymmOpenloop.dat}; 
			\addlegendentry{$\tilde{L}_k^{1}$ };
			\addplot[Black!50, very thick] table [x index=0, y index=8]{NonSymmOpenloop.dat}; 
			\addlegendentry{$\tilde{L}_k^{2}$ }; 
			\draw[<-, black] (5,3) to (5,2.7)node[below]{\eqref{eq:PrivateConst}};
			\draw[<-, black] (5,4) to (5,3.7)node[below]{\eqref{eq:PrivateConst}};
			\draw[-,black!60,dashed](0,4) to (60,4) node[xshift=-2.5cm,below,black]{$\tilde{c}^1$}; 
			\draw[-,black!60,dashed](0,3) to (60,3) node[xshift=-2cm,above,black]{$\tilde{c}^2$};
		\end{axis}
\end{tikzpicture}}
	\caption{Flow-rates/controls of the players (panel (a)), battery levels (panel (b)), aggregate flow-rates at the relay nodes (panel (c)) and destination nodes (panel (d)). Arrow with equation number pointing to a dashed dark line indicates that one of the constraints \eqref{eq:NFconstraints} is active.}  
	\label{fig:Fig3}
\end{figure}
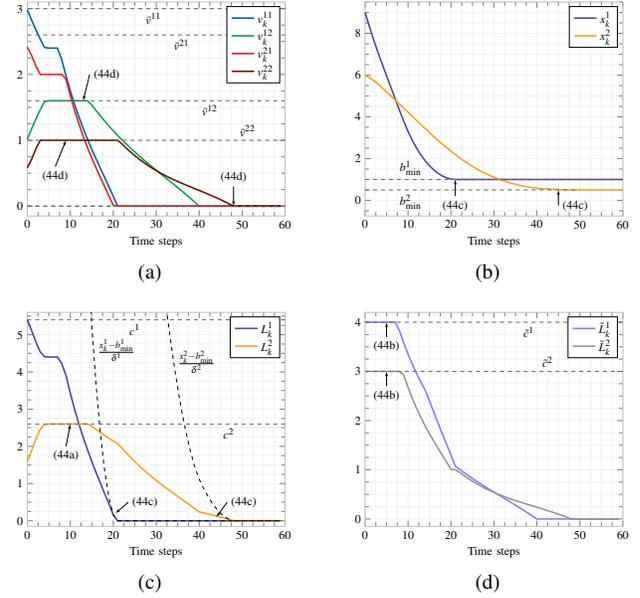
It is easily verified that the conditions required in Theorems \ref{th:OLNElcp} are satisfied for the chosen parameter values. We used the freely available PATH solver (available at {\url{https://pages.cs.wisc.edu/~ferris/path.html}}) to solve the LCP \eqref{eq:FinalOLNElcp}. 
The flow-rate constraints \eqref{eq:NFconstraints}, with the above parameter values are given by
\begin{align*}
		\text{Transmission:}\quad  & 0\leq v_k^{11}\leq 3,~0\leq v_k^{12}\leq 1.6\\
		&0\leq v_k^{21}\leq 2.6,~0\leq v_k^{22}\leq 1\\
		\text{Relay:}\quad & v_k^{11}+v_k^{21}\leq 5.4,~v^{12}_k+v_k^{22}\leq 2.6\\
		\text{Destination:}\quad & v_k^{11}+v_k^{12}\leq 4,~v_k^{21}+v_k^{22}\leq 3\\
		\text{Battery:}\quad & x_k^1-0.125(v_k^{11}+v_k^{21})\geq 1\\
		& x_k^2-0.075(v_k^{12}+v_k^{22})\geq 0.5.
\end{align*}
The transmission rate-dependent utility is concave, indicating that players receive higher payoffs by increasing their flow rates but are also incentivized to reduce flows due to the concave, increasing battery-level-dependent payoff. 
Figs. \ref{fig:Fig3} compare individual and aggregate flow rates, as well as relay battery levels, under GOLNE strategies. Initially, both players reduce flows through relay node $N_1$, which depletes faster than $N_2$ (see Fig. \ref{fig:fig21}). They then increase flows through $N_2$ to their peak rates, $v_k^{12} = 1.6$ and $v_k^{22} = 1$, which match the relay node's maximum capacity (see Figs. \ref{fig:fig21} and \ref{fig:fig23}). Since further transmission through $N_2$ is impossible, players maintain constant rates through $N_1$, constrained by destination capacities: $\tilde{c}^1 = v_k^{11} + v_k^{12} \rightarrow v_k^{11} = 4 - 1.6 = 2.4$ and $\tilde{c}^2 = v_k^{21} + v_k^{22} \rightarrow v_k^{21} = 3 - 1 = 2$ (see Figs. \ref{fig:fig21} and \ref{fig:fig23}). Aggregate flow rates through $N_1$ are not maximized because players self-limit to conserve battery life. In contrast, at $N_2$, capacity constraints are more restrictive than the incentive to conserve battery life, leading to a trade-off in the GOLNE flow rates. Players stop transmitting when relay nodes reach minimum battery levels: $N_1$ at $t = 21$ and $N_2$ at $t = 48$ (see Fig. \ref{fig:fig22}). 
\section{Conclusions}\label{sec:Conclusions}  
We studied a class of linear quadratic difference games with coupled-affine inequality constraints. We derived both the necessary and sufficient conditions for the existence of a generalized open-loop Nash equilibrium by establishing an equivalence between solutions of specific discrete-time linear complementarity systems and the convexity of players' objective functions. With additional assumptions, we demonstrated that GOLNE strategies can be obtained by solving a large-scale linear complementarity problem. As part of our future work, we plan to investigate the existence of a generalized feedback Nash equilibrium within this class of dynamic   games. 

\appendix
\subsubsection*{Notations for Section \ref{sec:Solvability}}
\label{sec:appendixnotation}
We define 
$\mathbf{R}_k=\oplus_{i=1}^{N}R_k^{i i}$,  
$\mathbf{B}_k=\oplus_{i=1}^{N}B_k^i$, $\mathbf{N}_k=\oplus_{i=1}^{N}[N^{i}_k]_i$, $\mathbf{M}_k=\oplus_{i=1}^{N}M^{i}_k$, $\mathbf{r}_k =\col{r_k^{i}}_{i=1}^{N}$, $\bu_k^{\star}=\col{u_k^{i\star}}_{i=1}^{N}$, $\bm{\upmu}_{k}^{\star}=\col{\mu_k^{i\star}}_{i=1}^{N}$, ~$k\in\K_l$,
$\mathbf{p}_k=\col{p_k^{i}}_{i=1}^{N}$,   
$P_k=\col{P_k^i}_{i=1}^{N}$, $\bm{\zeta}_{k}=\col{\zeta_{k}^{i}}_{i=1}^{N}$, $k\in\K$.  
\qquad 
\underline{In \eqref{eq:LCS5}:}  $\mathbf{G}_k=(\Lambda_{k})^{-1}A_{k}$, $[\overline{\mathbf{G}}_{k+1}]_i=-(\Lambda_{k})^{-1}B^i_k (R^{ii}_k)^{-1}{B^i_k}^{\prime}$, 
$[\tilde{\mathbf{G}}_{k}]_i=(\Lambda_{k})^{-1}B^i_k (R^{ii}_k)^{-1}[N^{i}_k]_i^{\prime}$,
$\mathbf{F}_k=-(\mathbf{R}_k)^{-1}\mathbf{B}_k^{\prime}P_{k+1}\mathbf{G}_k$, 
$\overline{\mathbf{F}}_{k+1}=-(\mathbf{R}_k)^{-1}\mathbf{B}_k^{\prime}(\mathbf{I}+P_{k+1} \overline{\mathbf{G}}_{k+1})$,  
$\tilde{\mathbf{F}}_k=(\mathbf{R}_k)^{-1}(\mathbf{N}_{k}^{\prime}-\mathbf{B}_k^{\prime}P_{k+1} \tilde{\mathbf{G}}_{k})$,
$\mathbf{H}_{k+1}=(\mathbf{I}_{N}\otimes A_k^{\prime})(\mathbf{I}+P_{k+1}\overline{\mathbf{G}}_{k+1})$, $\overline{\mathbf{H}}_{k}=(\mathbf{I}_{N}\otimes A_k^{\prime})P_{k+1}\tilde{\mathbf{G}}_{k}-\mathbf{M}_{k}^{\prime}$ for $k\in \K_l$.\qquad  
 \underline{For the state transition matrices in \eqref{eq:Forwardeq2}-\eqref{eq:Backwardeq2}:} $\phi(k, \tau)=\mathbf{G}_{k-1}\mathbf{G}_{k-2}\cdots\mathbf{G}_{\tau}$ for $k > \tau$ and $\phi(k, \tau)=\mathbf{I}$ for $k=\tau$, $k, \tau \in \K$. $\varphi(k, \tau)=\mathbf{H}_{k+1}\mathbf{H}_{k+2}\cdots\mathbf{H}_{\tau}$ for $k < \tau$ and $\varphi(k, \tau)=\mathbf{I}$  for $k=\tau$, $k, \tau \in \K_l$.\qquad 
\underline{In \eqref{eq:OLNEaggrEq}:}
$x_{\K}^{\star}=\col{x_{k}^{\star}}_{k=0}^{K-1}$, $\bu^{\star}_\K=\col{\bu^{\star}_k}_{k=0}^{K-1}$, $\bm{\upmu}^{\star}_\K=\col{\bm{\upmu}^{\star}_k}_{k=0}^{K-1}$,  $\mathbf{p}_\K=\col{\mathbf{p}_{k+1}}_{k=0}^{K-1}$, $\bm{\zeta}_\K=\col{\bm{\zeta}_{k+1}}_{k=0}^{K-1}$, $\textbf{r}_\K=\col{\textbf{r}_k}_{k=0}^{K-1}$, 
$\mathbf{F}_\K=\oplus_{k=0}^{K-1}\mathbf{F}_k$, $ 
\overline{\mathbf{F}}_\K=\oplus_{k=0}^{K-1}\overline{\mathbf{F}}_{k+1}$, $\tilde{\mathbf{F}}_\K=\oplus_{k=0}^{K-1}\tilde{\mathbf{F}}_k$, 
$[\bm{\Phi}_0]_{k}=\phi(k-1, 0)$, $[\bm{\Phi}_1]_{1 \tau}=\mathbf{0}$, $[\bm{\Phi}_2]_{1 \tau}=\mathbf{0}$,  $[\bm{\Phi}_1]_{k \tau}= \sum_{\rho=1}^{\mathrm{min}\,\,(k-1,\tau-1)}\phi(k-1, \rho)\overline{\mathbf{G}}_{\rho}\varphi(\rho, \tau-1)$ for $k>1$,  
$[\bm{\Phi}_2]_{k \tau}=\phi(k-1, 1) \tilde{\mathbf{G}}_{0}$ for $k>\tau=1$,
$[\bm{\Phi}_2]_{k \tau}=\big(\sum_{\rho=1}^{\tau-1}\phi(k-1, \rho)\overline{\mathbf{G}}_{\rho}\varphi(\rho, \tau-1)\big)\overline{\mathbf{H}}_{\tau-1}
+\phi(k-1, \tau)\tilde{\mathbf{G}}_{\tau-1}$ for $1<\tau<k$, 
$ [\bm{\Phi}_2]_{k \tau}=\big(\sum_{\rho=1}^{k-1}\phi(k-1, \rho)\overline{\mathbf{G}}_{\rho}\varphi(\rho, \tau-1)\big)\overline{\mathbf{H}}_{\tau-1}$ for $\tau \geq k$  
$[\bm{\Psi}_1]_{k \tau}=\varphi(k, \tau)$, $[\bm{\Psi}_2]_{k \tau}= \varphi(k, \tau-1)\overline{\mathbf{H}}_{\tau-1}$ for $\tau \geq k$, $[\bm{\Psi}_1]_{k \tau}=\mathbf{0}$, $[\bm{\Psi}_2]_{k \tau}=\mathbf{0}$ for $\tau < k$ with $ k,\tau \in \K_r$, $\mathsf{M}_\K=\oplus_{k=0}^{K-1}(\overline{\mathbf{M}}_k+\overline{\mathbf{N}}_k\mathbf{F}_{k})$, $\mathsf{N}_\K=\oplus_{k=0}^{K-1}(\overline{\mathbf{N}}_k\overline{\mathbf{F}}_{k+1})$ and $\tilde{\mathsf{N}}_\K=\oplus_{k=0}^{K-1}(\overline{\mathbf{N}}_k\tilde{\mathbf{F}}_{k})$.
 
\bibliographystyle{IEEEtran}
\bibliography{CLSupdated} 
 
\end{document}